\documentclass[oneside,english]{amsart}
\usepackage[T1]{fontenc}
\usepackage[latin9]{inputenc}
\usepackage{babel}
\usepackage{amstext}
\usepackage{amsthm}
\usepackage{amssymb}
\usepackage[pdfusetitle,
 bookmarks=true,bookmarksnumbered=false,bookmarksopen=false,
 breaklinks=false,pdfborder={0 0 0},pdfborderstyle={},backref=false,colorlinks=false]
 {hyperref}

\makeatletter
\numberwithin{equation}{section}
\numberwithin{figure}{section}
\theoremstyle{plain}
\newtheorem{thm}{\protect\theoremname}[section]
\theoremstyle{plain}
\newtheorem{prop}[thm]{\protect\propositionname}
\theoremstyle{remark}
\newtheorem{rem}[thm]{\protect\remarkname}
\theoremstyle{plain}
\newtheorem{lem}[thm]{\protect\lemmaname}
\theoremstyle{definition}
\newtheorem{defn}[thm]{\protect\definitionname}
\theoremstyle{remark}
\newtheorem{claim}[thm]{\protect\claimname}

\makeatother

\providecommand{\claimname}{Claim}
\providecommand{\definitionname}{Definition}
\providecommand{\lemmaname}{Lemma}
\providecommand{\propositionname}{Proposition}
\providecommand{\remarkname}{Remark}
\providecommand{\theoremname}{Theorem}

\begin{document}
\global\long\def\Q{\mathbf{\mathbb{Q}}}%
\global\long\def\R{\mathbf{\mathbb{R}}}%
\global\long\def\C{\mathbf{\mathbb{C}}}%
\global\long\def\Z{\mathbf{\mathbb{Z}}}%
\global\long\def\N{\mathbf{\mathbb{N}}}%
\global\long\def\T{\mathbb{T}}%
\global\long\def\Im{\mathrm{Im}}%
\global\long\def\Re{\mathrm{Re}}%
\global\long\def\H{\mathcal{H}}%
\global\long\def\M{\mathbb{M}}%
\global\long\def\P{\mathbb{P}}%
\global\long\def\L{\mathcal{L}}%
\global\long\def\F{\mathcal{\mathcal{F}}}%
\global\long\def\s{\sigma}%
\global\long\def\Rc{\mathcal{R}}%
\global\long\def\W{\tilde{W}}%

\global\long\def\G{\mathcal{G}}%
\global\long\def\d{\partial}%
 
\global\long\def\jp#1{\langle#1\rangle}%
\global\long\def\norm#1{\|#1\|}%
\global\long\def\mc#1{\mathcal{\mathcal{#1}}}%

\global\long\def\Right{\Rightarrow}%
\global\long\def\Left{\Leftarrow}%
\global\long\def\les{\lesssim}%
\global\long\def\hook{\hookrightarrow}%
\global\long\def\weak{\rightharpoonup}%
\global\long\def\supp{\mathrm{supp}}%
\global\long\def\sub{\mathrm{sub}}%
\global\long\def\loc{\mathrm{loc}}%

\global\long\def\D{\mathbf{D}}%
\global\long\def\rad{\mathrm{rad}}%

\global\long\def\env{\mathrm{Env}}%
\global\long\def\re{\mathrm{re}}%
\global\long\def\im{\mathrm{im}}%
\global\long\def\err{\mathrm{err}}%

\global\long\def\d{\partial}%
 
\global\long\def\jp#1{\langle#1\rangle}%
\global\long\def\norm#1{\|#1\|}%
\global\long\def\ol#1{\overline{#1}}%
\global\long\def\wt#1{\widehat{#1}}%
\global\long\def\tilde#1{\widetilde{#1}}%

\global\long\def\br#1{(#1)}%
\global\long\def\Bb#1{\Big(#1\Big)}%
\global\long\def\bb#1{\big(#1\big)}%
\global\long\def\lr#1{\left(#1\right)}%

\global\long\def\ve{\varepsilon}%
\global\long\def\la{\lambda}%
\global\long\def\al{\alpha}%
\global\long\def\be{\beta}%
\global\long\def\ga{\gamma}%
\global\long\def\La{\Lambda}%
\global\long\def\De{\Delta}%
\global\long\def\na{\nabla}%

\global\long\def\ep{\epsilon}%
\global\long\def\fl{\flat}%
\global\long\def\sh{\sharp}%
\global\long\def\calN{\mathcal{N}}%
\global\long\def\avg{\mathrm{avg}}%
\global\long\def\rc{\mathrm{rc}}%
\title[Global well-posedness of NLS on $\T^{d}$]{global well-posedness of the energy-critical nonlinear Schr\"odinger
equations on $\T^{d}$}
\author{Beomjong Kwak}
\email{beomjong@kaist.ac.kr}
\address{Department of Mathematical Sciences, Korea Advanced Institute of Science
and Technology, 291 Daehak-ro, Yuseong-gu, Daejeon 34141, Korea}
\begin{abstract}
In this paper, we prove the global well-posedness of the energy-critical
nonlinear Schr\"odinger equations on the torus $\T^{d}$ for general
dimensions. This result is new for dimensions $d\ge5$, extending
previous results for $d=3,4$ \cite{ionescu2012energy,YUE2021754}.
Compared to the cases $d=3,4$, the regularity theory for higher $d$,
developed in the underlying local well-posedness result \cite{LWP(working)},
is less understood. In particular, stability theory and inverse inequalities,
which are ingredients in \cite{ionescu2012energy,YUE2021754} and
more generally in the widely used concentration compactness framework
since \cite{kenig2006global}, are too weak to be applied to higher
dimensions.

Our proof introduces a new strategy for addressing global well-posedness
problems. Without relying on perturbation theory, we develop tools
to analyze the concentration dynamics of the nonlinear flow. On the
way, we show the formation of a nontrivial concentration.
\end{abstract}

\maketitle

\section{Introduction}

The subject of this paper is the global well-posedness of the Cauchy
problem for the energy-critical nonlinear Schr\"odinger equation (NLS)
on the periodic space

\begin{equation}
\begin{cases}
iu_{t}+\De u=\pm|u|^{\frac{4}{d-2}}u=:\mathcal{N}(u)\\
u(0)=u_{0}\in H^{1}(\T^{d})
\end{cases},\tag{{\text{\text{NLS}\ensuremath{(\mathbb{T}^{d}})}}}\label{eq:NLS T^d}
\end{equation}
where $u:\R\times\T^{d}\rightarrow\C$, $d\ge3$ and $\T^{d}=\R^{d}/(\theta_{1}\Z\times\cdots\times\theta_{d}\Z),\theta_{j}>0$
is any rectangular torus.

For $d\ge3$, the following local well-posedness result for (\ref{eq:NLS T^d})
was shown in \cite{herr2011global,killip2016scale,LWP(working)}:\footnote{Indeed, the critical local well-posedness was shown over a larger
regime of mass-supercritical power in \cite{lee2019local,LWP(working)},
though we focus on the energy-critical case.}
\begin{prop}[\cite{herr2011global,killip2016scale,LWP(working)}]
\label{prop:LWP}For $d\ge3$, (\ref{eq:NLS T^d}) is locally well-posed
in $C^{0}H^{1}\cap Y^{1}$. Here, $Y^{1}$ is the atomic-based space
introduced in \cite{herr2011global}; see Section \ref{subsec:Schr=0000F6dinger-operators,-Strichart}.

Nevertheless, the flow map fails to be Lipschitz for $d\ge7$.
\end{prop}

Although the local well-posedness has been shown for every dimension,
the local analysis is less understood and behaves worse for $d\ge5$. 

The local well-posedness result for $d\ge5$ in \cite{LWP(working)}
involved nonlinear estimates with limited rooms for adjustments, based
on function spaces newly introduced there.

The large-data global well-posedness of (\ref{eq:NLS T^d}) was shown
for $d=3,4$ \cite{ionescu2012energy,YUE2021754} and previously unknown
for $d\ge5$. For $d\ge5$, the nonlinearity is non-algebraic and
so the nonlinear part is not sufficiently decomposable. The flow map
is only known to be continuous (and the failure of H\"older continuity
was also shown in \cite{LWP(working)}). For dimensions $d=3,4$,
the results \cite{ionescu2012energy,YUE2021754} were shown via the
frameworks of \cite{kenig2006global} and \cite{ionescu2012curved},
based on strong stability theory and decomposability of the nonlinearity.

In this paper, we show the large-data global well-posedness of (\ref{eq:NLS T^d})
for higher dimensions $d$. (\ref{eq:NLS T^d}) enjoys the energy
and mass conservation laws:
\[
E(u)=\frac{1}{2}\int_{\T^{d}}\left|\na u\right|^{2}\pm\frac{1}{\frac{4}{d-2}+2}\left|u\right|^{\frac{4}{d-2}+2}dx
\]
and
\[
M(u)=\int_{\T^{d}}\left|u\right|^{2}dx.
\]
Considering scaling, (\ref{eq:NLS T^d}) is energy-critical. Depending
on the sign of $\mc N$, (\ref{eq:NLS T^d}) is defocusing if $\calN(u)=\left|u\right|^{\frac{4}{d-2}}u$
and focusing if $\calN(u)=-\left|u\right|^{\frac{4}{d-2}}u$.

We now state our main results. For the defocusing case, we show the
following:
\begin{thm}
\label{thm:defocusing GWP}Let $d\ge3$. (\ref{eq:NLS T^d}) with
the defocusing sign is globally well-posed.
\end{thm}

Theorem \ref{thm:defocusing GWP} is new for $d\ge5$, as mentioned
above.

For the focusing case, there is a threshold for global existence,
analogous to that on $\R^{d}$ (Proposition \ref{prop:focusing R^d}).
We adopt a modified energy introduced in \cite{YUE2021754}:
\begin{equation}
E_{c_{0}}(u):=E(u)+\frac{c_{0}}{2}\norm u_{L^{2}}^{2},\label{eq:E*}
\end{equation}

where $c_{0}=c_{0}(\T^{d})>0$ is the optimal constant from a Sobolev
embedding on $\T^{d}$; see Proposition \ref{prop:energy trapping}.

Using the energy trapping lemma developed in \cite{YUE2021754}, we
obtain the analogous result for focusing equations.
\begin{thm}
\label{thm:focusing GWP}Let $d\ge4$. Consider (\ref{eq:NLS T^d})
with the focusing sign. Let $u_{0}\in H^{1}(\T^{d})$ be data such
that
\[
\norm{\na u_{0}}_{L^{2}}^{2}+c_{0}\norm{u_{0}}_{L^{2}}^{2}<\norm{\na W}_{L^{2}}^{2}
\]
and
\[
E_{c_{0}}(u_{0})<E(W),
\]
where $W$ denotes the ground state of the energy-critical focusing
NLS on $\R^{d}$
\begin{equation}
W=\left(1+\frac{1}{d(d-2)}\left|x\right|^{2}\right)^{-\frac{d-2}{2}}.\label{eq:W}
\end{equation}
Then, there exists a unique global solution $u:\R\times\T^{d}\rightarrow\C$
to (\ref{eq:NLS T^d}).
\end{thm}

\begin{rem}
For $d=3$, Theorem \ref{thm:focusing GWP} remains open. This is
an essential challenge; the analogous GWP on $\R^{3}$ is still open
for general data.
\end{rem}

The energy-critical GWP of NLS on $\R^{d}$ has been extensively studied.
For the radial defocusing case, Bourgain \cite{bourgain1999global}
and Grillakis \cite{grillakis2000nonlinear} showed GWP for the defocusing
NLS in dimension $d=3$ under radial symmetry, where concentration
of a solution was observed. Since then, the defocusing GWP for $d\ge3$
has been shown for general nonradial data in \cite{colliander2008global,ryckman2007global,visan2007defocusing}.

Afterward, the concentration compactness method was introduced by
Kenig and Merle \cite{kenig2006global}. They showed GWP of the energy-critical
focusing NLS for $3\le d\le5$ under the radial symmetry, by reducing
the GWP of a general solution to that of an almost periodic solution
modulo symmetries. This approach was extended to general nonradial
data for $d\ge4$ \cite{killip2010focusing,dodson2019global}.

At the level of regularity theory, a key ingredient for the concentration
compactness roadmap is the profile decomposition, i.e., the approximation
of a solution by a sum of rescaled solutions. For this, the following
can be seen as conventional ingredients:

\medskip{}

(1) a strong (e.g., Lipschitz) stability theory,

(2) an inverse inequality to capture concentration,

(3) smallness of nonlinear interactions between profiles.

\medskip{}

In \cite{ionescu2012curved}, such profile decomposition analysis
was extended to curved domains. Once (1)--(3) above are provided,
the approach in \cite{ionescu2012curved} works robustly by approximating
solutions on a domain as spatial cutoffs of solutions on $\R^{d}$.
For dimensions $d=3,4$, where the nonlinearity $\mc N(u)=\pm|u|^{\frac{4}{d-2}}u$
is algebraic, (1)--(3) are fulfilled on $\T^{d}$ by the previous
achievements in Strichartz estimates \cite{bourgain2015proof,killip2016scale}
and function spaces \cite{herr2011global}. Consequently, the argument
in \cite{ionescu2012curved} was successfully applied to showing the
first large-data GWP results for the energy-critical defocusing NLS
on periodic domains: the rational torus $\T^{3}$ \cite{ionescu2012energy}
and cylinders $\T^{3}\times\R$ \cite{ionescu2012global}. Later,
Yue \cite{YUE2021754} suggested modified energies for the focusing
nonlinearity and showed sub-threshold GWP results on $\T^{4}$.

For $d\ge5$, where the nonlinearity is not algebraic, the regularity
behavior on $\T^{d}$ is less understood:

\medskip{}

(i) LWP has barely been shown and Lipschitz well-posedness fails for
$d\ge7$ (see \cite[Theorem 1.3]{LWP(working)}).

(ii) On the solution space of LWP, the inverse property only partially
holds (i.e., only within a frequency localization $P_{N}$; see Proposition
\ref{prop:inverse embedding}).

(iii) Quantitative analysis of decompositions of the nonlinearity
is less obvious due to the non-algebraic nonlinearity.

\medskip{}

From (i)--(iii), what can be directly said about a blow-up solution
is only the existence of a possibly unbounded number of concentrations
of different scales whose sizes sum up to $\epsilon>0$. In particular,
the conventional profile decomposition is not obvious.

In proving Theorem \ref{thm:defocusing GWP} and Theorem \ref{thm:focusing GWP},
we intend to present a robust argument for the global well-posedness.
In the limited circumstances where (1)--(3) fail, we reconstruct
the concentration compactness roadmap for $d\ge5$ using an approach
different from the profile decomposition method.

We outline the strategy for our proof. First, we capture spacetime
concentrations and limit profiles of nonlinear solutions using weak
limits. The effect of a weak limit profile on the dynamics of the
original solution on $\T^{d}$ is observed as scattering-like behavior
in a distributional sense. More precisely, we show that a bounded
sequence of solutions on either $\T^{d}$ or $\R^{d}$ converges uniformly
to the sequence of their scattering limits as $t\rightarrow-\infty$
in the distributional sense.

Then, we upgrade the frequency-localized linear inverse theorem to
concentration of the nonlinear flow (Proposition \ref{prop:amp}).
This is a critical step in this work; we prove that the solution does
not fully split into small-amplitude bubbles of multiple scales. We
formalize the following heuristic: if a solution splits into small-amplitude
bubbles distinguishable in a distributional sense, a maximum-amplitude
bubble evolves in some sense like a free evolution.

With the aforementioned tools, we prove our main results, Theorem
\ref{thm:defocusing GWP} and Theorem \ref{thm:focusing GWP}. Here,
the main difficulty is that we only know an $\epsilon$-amount of
concentration and have no information about the remainder. We do not
use the conventionally used concept of orthogonalities between profiles.
Instead, we view a certain triple weak limit in two different perspectives
for a proof by contradiction. (For details, see the proof of Proposition
\ref{prop:main}).

Our proofs of Theorem \ref{thm:defocusing GWP} and Theorem \ref{thm:focusing GWP}
presume the global well-posedness results on $\R^{d}$, as these are
well-known. Nevertheless, our method can also be applied to carry
out the concentration compactness roadmap on $\R^{d}$ itself, up
to the reduction to almost periodic solutions modulo symmetries. This
can be seen as a reproof of the global well-posedness on $\R^{d}$,
alternative to the profile decomposition approach.\footnote{Note that the previous proof of GWP on $\R^{d}$ for $d\ge7$ was
based on H\"older continuity of the solution map shown in \cite{tao2005stability}.
\cite{tao2005stability} was a nontrivial result by itself; for instance,
a nontrivial package of Strichartz estimates that holds on $\R^{d}$
(called exotic Strichartz in \cite{tao2005stability}) was used.}

This issue is briefly noted in Appendix A.

\medskip{}
In Section \ref{sec:Preliminaries}, we provide preliminary materials.
In Section \ref{sec:Basic-properties-of}, we provide primary concepts
and basic facts on profiles and weak limits. In Section \ref{sec:Solution-spaces-and},
we define a solution norm and prove preparatory facts related to it.
In Section \ref{sec:Global-well-posedness-of NLS}, we develop the
main arguments and prove Theorem \ref{thm:defocusing GWP} and Theorem
\ref{thm:focusing GWP}.

\subsection*{Acknowledgements}

The author appreciates Soonsik Kwon for helpful communications and
encouragements. The author is partially supported by National Research
Foundation of Korea, RS-2019-NR040050 and NRF-2022R1A2C1091499.

Also, a part of this work was done while the author was visiting Bielefeld
University through International Research Training Group 2235. He
appreciates their kind hospitality.

\section{\label{sec:Preliminaries}Preliminaries}

\subsection{Notations}

We denote $A\les B$ if $A\le CB$ holds for some constant $C$.

Given an interval $I\subset\R$, we denote by $\chi_{I}$ the sharp
cutoff of $I$.

For simplicity, throughout this paper we assume our domain is the
square torus $\T^{d}=\R^{d}/(2\pi\Z)^{d}$. Since we do not rely on
specific properties of the square torus, such as number-theoretic
arguments on frequencies, the proof of Theorem \ref{thm:defocusing GWP}
and Theorem \ref{thm:focusing GWP} in this paper stays valid on any
rectangular torus.

We denote $a:=\frac{4}{d-2}$ (i.e., $\mc N(u)=\pm|u|^{a}u$).

\subsubsection*{Fourier truncations}

We handle functions of spacetime variables $f(t,x)$ and $f(x)$ for
$x\in\T^{d}$ and $t\in\R$. We denote the Fourier series of $f$
with the associated variable $x$ by either $\F_{x}f$ or $\widehat{f}$.

We use Littlewood-Paley projection operators. Denote the set of natural
numbers by $\N=\left\{ 0\right\} \cup\Z_{+}$ and dyadic numbers by
$2^{\N}$. Let $\psi:\R\rightarrow[0,\infty)$ be a smooth even bump
function such that $\psi|_{[-1,1]}\equiv1$ and $\text{supp}(\psi)\subset[-\frac{11}{10},\frac{11}{10}]$.
For a dyadic number $N\in2^{\N}$, we denote by $\psi_{N}$ the function
$\psi_{N}(\xi)=\psi(\xi/N)$ and $P_{\le N}$ the Fourier multiplier
induced by $\psi_{N}(\xi_{1})\cdots\psi_{N}(\xi_{d})$. We denote
$P_{N}:=P_{\le N}-P_{\le N/2}$ with the convention $P_{\le2^{-k}}:=0$
for $k>0$. For simplicity, we use abridged notations $u_{N}:=P_{N}u$
and $u_{\le N}:=P_{\le N}u$ for a function $u:\T^{d}\rightarrow\C$.

We also use time Fourier projections, denoted with the superscript
$t$. For each $N\in2^{\Z}$, we denote by $P_{\le N}^{t}$ the Fourier
projection for the time variable $t$ with the multiplier $\psi_{N}(\cdot)$.

Analogous to $\T^d$, we use smooth Littlewood-Paley operators on $\R^d$. We use the same notation, except that we allow $N\in 2^\Z$, i.e., we consider the homogeneous Littlewood-Paley cutoffs for $\R^d$. 

\subsection{Function spaces}

In this subsection, we collect function space notations of functions
of variables $x\in\T^{d}$ and $(t,x)\in\R\times\T^{d}$. These spaces
are either $L^{p}$-based or generated from atomic spaces. For a Banach
space $B$, we denote its norm by $\norm{\cdot}_{B}$. When $B$ is
defined on the domain $\R\times\T^{d}$, we denote by $B'$ its dual
space with respect to the inner product $\left\langle u,v\right\rangle _{L^{2}(\R\times\T^{d})}:=\int_{\R\times\T^{d}}\overline{u}vd(t,x)$.

We denote by $\mc S(\R^{d})$ and $\mc S'(\R^{d})$ the Schwartz space
and the space of tempered distributions on $\R^{d}$.

For $q\in[1,\infty]$ and a Banach space $E$ defined on $\T^{d}$,
we denote by $L^{q}E=L_{t}^{q}E$ the mixed norm
\[
\norm u_{L^{q}E}:=\left(\int_{\R}\norm{u(t)}_{E}^{q}dt\right)^{1/q}.
\]

We use Besov-type norms for the time variable. Let $s\in\R$ and $p,q\in[1,\infty]$.
For a Banach space $E$ defined on either $\T^{d}$ or $\R^{d}$,
we denote by $B_{p,q}^{s}E=(B_{p,q}^{s})_{t}E_{x}$ the (function-valued)
Besov space defined as the dyadic summation on time frequency cutoffs
\[
\|u\|_{B_{p,q}^{s}E}:=\left(\sum_{N\in2^{\N}}N^{qs}\|P_{N}^{t}u\|_{L^{p}E}^{q}\right)^{1/q}+\|P_{\le1}^{t}u\|_{L^{p}E}.
\]
We also use norms concerning $\ell^{q}$-summations of spatial frequency-localized
pieces. Let $E$ be a function space defined on $\R\times\T^{d}$.
We denote 
\[
\|u\|_{\ell_{s}^{q}E}:=\left(\sum_{N\in2^{\N}}N^{qs}\|u_{N}\|_{E}^{q}\right)^{1/q}+\|u_{\le1}\|_{E}.
\]
When $s=0$, the subscript $s$ in $\ell_{s}^{q}$ is omitted and
we simply write it as $\ell^{q}$.
\begin{lem}
Let $d\in\N$, $p\in(1,\infty)$, $\theta\in(0,1)$, and $s_{0}\ne s_{1}$
be real numbers. Let $s_{\theta}=(1-\theta)s_{0}+\theta s_{1}$. For
any function $f\in B_{p,\infty}^{s_{0}}L^{p}\cap B_{p,\infty}^{s_{1}}L^{p}$
on $\R\times\T^{d}$, we have
\begin{equation}
\norm f_{B_{p,1}^{s_{\theta}}L^{p}}\les_{s_{0},s_{1},\theta,p}\norm f_{B_{p,\infty}^{s_{0}}L^{p}}^{1-\theta}\norm f_{B_{p,\infty}^{s_{1}}L^{p}}^{\theta}.\label{eq:Besov real interp E s0s1}
\end{equation}
\end{lem}

\begin{proof}
Equivalently, it suffices to show for a positive function $g:2^{\N}\rightarrow\R$
that
\begin{equation}
\sum_{N\in2^{\N}}N^{s_{\theta}}g(N)\les_{s_{0},s_{1},\theta,p}\left(\sup_{N\in2^{\N}}N^{s_{0}}g(N)\right)^{1-\theta}\cdot\left(\sup_{N\in2^{\N}}N^{s_{1}}g(N)\right)^{\theta}=:A^{1-\theta}B^{\theta}.\label{eq:interpolation proof 1}
\end{equation}
(\ref{eq:interpolation proof 1}) follows from a dyadic summation
on the inequality $g(N)\le\min\left\{ AN^{-s_{0}},BN^{-s_{1}}\right\} $.
\end{proof}
Every weak limit we consider in this paper is in the distribution
sense. Given a sequence of functions $\left\{ f_{n}\right\} $, we
denote by $\lim_{n}f_{n}$ the weak limit of $f_{n}$ (if it exists).

For two function norms $\norm{\cdot}_{B_{1}}$ and $\norm{\cdot}_{B_{2}}$,
we denote by $B_{1}\cap B_{2}$ and $B_{1}+B_{2}$ the norms

\[
\|f\|_{B_{1}\cap B_{2}}:=\|f\|_{B_{1}}+\norm f_{B_{2}}
\]
and
\[
\norm f_{B_{1}+B_{2}}:=\inf_{f_{1}+f_{2}=f}\norm{f_{1}}_{B_{1}}+\norm{f_{2}}_{B_{2}}.
\]

\subsection{Schr\"odinger operators, Strichartz estimates, and atomic spaces\label{subsec:Schr=0000F6dinger-operators,-Strichart}}

In this subsection, we collect facts related to the Schr\"odinger operator.
More specifically, we present Strichartz estimates for the linear
Schr\"odinger operator on tori, the atomic spaces, and the Galilean
transform.

We begin with notations related to the linear Schr\"odinger operator.
For a function $\phi:\T^{d}\rightarrow\C$ and time $t\in\R$, we
denote by $e^{it\De}\phi(x)$ the linear evolution

\[
e^{it\De}\phi=\mathcal{F}_{x}^{-1}(e^{-it\left|\xi\right|^{2}}\wt{\phi}).
\]

We denote by $K^{+}$ the retarded Schr\"odinger operator; for a function
$f:\R\times\T^{d}\rightarrow\C$, $K^{+}f$ is given as
\begin{equation}
K^{+}f(t):=-i\int_{-\infty}^{t}e^{i(t-s)\De}f(s)ds.\label{eq:K+}
\end{equation}

\subsubsection*{Strichartz estimates and atomic spaces}

We state a kernel estimate for the Schr\"odinger operator on tori. This
was first shown by Bourgain in \cite{bourgain1993fourier}.
\begin{prop}[Kernel bound, \cite{bourgain1993fourier,killip2016scale}]
\label{prop:kernel bound bourgain}Consider a rectangular torus $\T^{d}=\R^{d}/(\theta_{1}\Z\times\cdots\times\theta_{d}\Z)$
with $\theta_{1},\ldots,\theta_{d}>0$. Let $(a_{j},q_{j})$ be a
pair of coprime integers for $j=1,\ldots,d$, $t\in\R$, and $N\in2^{\N}$
be a dyadic number such that
\[
1\le q_{j}<N\text{ and }\left|\frac{2\pi t}{\theta_{j}^{2}}-\frac{a_{j}}{q_{j}}\right|\le\frac{1}{q_{j}N}.
\]
Then, we have
\begin{equation}
\norm{e^{it\De}\delta_{N}}_{L^{\infty}(\T^{d})}\les\prod_{j=1}^{d}\frac{N}{\sqrt{q_{j}}\left(1+N\left|\frac{2\pi t}{\theta_{j}^{2}}-\frac{a_{j}}{q_{j}}\right|^{1/2}\right)}.\label{eq:Bourgain bound}
\end{equation}
\end{prop}

\begin{proof}
(\ref{eq:Bourgain bound}) is given in \cite[Lemma 2.2]{killip2016scale}
for $0\le a_{j}\le q_{j}$. Due to the time-periodicity of $e^{it\d_{x}^{2}}$,
(\ref{eq:Bourgain bound}) holds for general $a_{j}\in\Z$.
\end{proof}
Although (\ref{eq:Bourgain bound}) depends on the ratio of side lengths
of $\T^{d}$, the dependence is irrelevant to our paper. Indeed, (\ref{eq:Bourgain bound})
is used only in the proof of the extinction lemma (\ref{eq:kernel 0}).
In there, the side lengths $\theta_{j}$ are irrelevant. That is the
only place where the rationality of side lengths is involved. For
this reason, one can verify that the proof is valid for any torus.
Hence, for simplicity of presentation, we assume that our domain $\T^{d}$
is a square torus.

The next proposition is a scale-invariant Strichartz estimate for
general tori. For the rational torus, this result follows from the
$\ell^{2}$-decoupling theorem by Bourgain and Demeter \cite{bourgain2015proof}.
For general tori, the subcritical version was first shown in \cite{bourgain2015proof},
then was sharpened to the critical scale in \cite{killip2016scale}.
\begin{prop}[\cite{bourgain1993fourier,bourgain2015proof,killip2016scale}]
Let $p\in(\frac{2(d+2)}{d},\infty)$ and $\s=\frac{d}{2}-\frac{d+2}{p}>0$.
Let $I\subset\R$ be a finite interval. Then, we have
\begin{equation}
\sup_{N\in2^{\N}}\|P_{N}e^{it\De}f\|_{L_{t,x}^{p}(I\times\T^{d})}\les_{p,I}\|f\|_{H^{\s}(\T^{d})}.\label{eq:Bourgain Strichartz}
\end{equation}
\end{prop}

Next, we recall the definition of the function space $Y^{s}$ first
appeared in \cite{herr2011global}, which is the solution norm of
LWP (Proposition \ref{prop:LWP}). For a general theory, we refer
to \cite{koch2014dispersive}, \cite{herr2011global}, and \cite{hadac2009well}.
\begin{defn}[\cite{herr2011global}]

Let $\mathcal{Z}$ be the collection of finite non-decreasing sequences
$\left\{ t_{k}\right\} _{k=0}^{K}$ in $(-\infty,\infty]$. Let $V_{\rc}^{2}$
be the space of right-continuous functions $f:\R\rightarrow\C$ satisfying
$\lim_{t\rightarrow-\infty}f(t)=0$, equipped with the norm

\[
\|f\|_{V_{\rc}^{2}}:=\left(\sup_{\left\{ t_{k}\right\} _{k=0}^{K}\in\mc Z}\sum_{k=1}^{K}\left|f(t_{k})-f(t_{k-1})\right|^{2}\right)^{1/2},
\]
where the convention $f(\infty)=0$ is used. For $s\in\R$, we define
$Y^{s}$ as the space of $u:\R\times\T^{d}\rightarrow\C$ such that
$\widehat{u}(\xi)$ lies in $V_{\rc}^{2}$ for each $\xi\in\Z^{d}$
and has a finite norm

\[
\|u\|_{Y^{s}}:=\left(\sum_{\xi\in\Z^{d}}\jp{\xi}^{2s}\|e^{it|\xi|^{2}}\widehat{u(t)}(\xi)\|_{V_{\rc}^{2}}^{2}\right)^{1/2}<\infty.
\]
While the space $Y^{s}$ is defined on the full-time domain $\R\times\T^{d}$,
since Strichartz estimates such as (\ref{eq:Bourgain Strichartz})
depend on the size of the time interval, we often restrict the time
interval to a finite interval $I\subset\R$, and denote this restriction
by $Y^{s}(I)$.
\end{defn}

\begin{prop}[Strichartz estimates, \cite{herr2011global,hadac2009well,hadac2010erratum}]
\label{prop:atomic space props}Let $p\in(\frac{2(d+2)}{d},\infty)$
and $\s=\frac{d}{2}-\frac{d+2}{p}$. For $s\in\R$, we have the following
estimates:
\begin{enumerate}
\item $Y^{s}=\ell^{2}Y^{s}\hook L^{\infty}H^{s}$.
\item For $f\in Y^{s}$, $f$ is right-continuous in $H^{s}(\T^{d})$ and
for $t_{0}\in\R$, $\lim_{t\rightarrow t_{0}^{-}}f(t)$ exists in
$H^{s}(\T^{d})$.
\item Let $I\subset\R$ be a finite interval. For $N\in2^{\N}$, we have
the estimate
\begin{equation}
\|\chi_{I}P_{\le N}u\|_{L_{t,x}^{p}}\les_{p,I}N^{\s}\|P_{\le N}u\|_{Y^{0}}.\label{eq:cube Strichartz}
\end{equation}
\item For a function $f:\R\times\T^{d}\rightarrow\C$, we have
\begin{equation}
\norm{K^{+}f}_{Y^{s}}\les\norm f_{(Y^{-s})'}.\label{eq:U2V2}
\end{equation}
\end{enumerate}
\end{prop}

(1) and (2) follow directly from the definition of $Y^{s}$. For (2),
the right-continuity and the existence of a left limit follow from
those of $V_{\rc}^{2}$; see \cite[Proposition 2.4]{hadac2009well}
for the proof.

(\ref{eq:cube Strichartz}) is a consequence of (\ref{eq:Bourgain Strichartz}),
used for example in \cite{killip2016scale}. It is obtained using
the atomic structure of $Y^{0}$-norm.

(\ref{eq:U2V2}) is a version of the $U^{2}-V^{2}$ dual estimate
used in \cite{herr2011global}.\footnote{In other literature such as \cite{herr2011global}, \cite{killip2016scale},
\cite{lee2019local}, (\ref{eq:U2V2}) is obtained with the left-hand
side replaced by the $U^{2}$-based norm $X^{s}$, using the duality
of atomic spaces. Since $X^{s}\hook Y^{s}$, that implies (\ref{eq:U2V2}).}

\subsubsection*{Galilean transforms}

For $\xi\in\Z^{d}$, we denote by $I_{\xi}:\mathcal{S}'(\R\times\T^{d})\rightarrow\mathcal{S}'(\R\times\T^{d})$
the Galilean transform mapping $u:\R\times\T^{d}\rightarrow\C$ to
\begin{equation}
I_{\xi}u(t,x)=e^{ix\cdot\xi-it|\xi|^{2}}u(t,x-2t\xi).\label{eq:Galilean}
\end{equation}

\subsection{Global well-posedness of NLS on $\protect\R^{d}$}

We recall the global well-posedness results and scattering norm bounds
for NLS on $\R^{d},d\ge3$:
\begin{equation}
\begin{cases}
iu_{t}+\De u=\pm|u|^{\frac{4}{d-2}}u=:\mathcal{N}(u)\\
u(0)=u_{0}\in\dot{H}^{1}(\R^{d})
\end{cases}\tag{{\text{\text{NLS}\ensuremath{(\mathbb{\R}^{d}})}}}.\label{eq:NLS R^d}
\end{equation}
The following is a preparatory usual fact that is at the level of
local theory.
\begin{lem}
\label{lem:local theory}Let $u\in L^{\infty}\dot{H}^{1}\cap L_{t,x}^{\frac{2d+4}{d-2}}(I\times\R^{d}),I\subset\R$
be a Duhamel solution to (\ref{eq:NLS R^d}), either focusing or defocusing.
Then, $u\in C^{0}\dot{H}^{1}\cap L^{2}\dot{W}^{1,2_{1}}(I\times\R^{d})$
also holds and
\[
\norm u_{L^{2}\dot{W}^{1,2_{1}}(I\times\R^{d})}\les_{\norm u_{L^{\infty}\dot{H}^{1}\cap L_{t,x}^{\frac{2d+4}{d-2}}}}1.
\]
Here, we denoted $2_{s}=(\frac{1}{2}-\frac{s}{d})^{-1}$, with which
$L^{2_{s}}(\R^{d})$ and $\dot{H}^{s}(\R^{d})$ have the same scaling.
Furthermore, $u$ scatters both forward and backward in time, i.e.,
there exist $\phi_{\pm}\in\dot{H}^{1}(\R^{d})$ such that $\lim_{t\rightarrow\pm\infty}\norm{u(t)-e^{it\De}\phi_{\pm}}_{\dot{H}^{1}}=0$.
\end{lem}

\begin{proof}
This is a standard consequence of conventional Strichartz estimates;
for a direct reference (which is fairly stronger), see \cite[Theorem 1.3]{tao2005stability}.
\end{proof}
For the defocusing case, the following global well-posedness and scattering
result is known for arbitrary dimensions $d\ge3$:
\begin{prop}[\cite{colliander2008global,ryckman2007global,visan2007defocusing}]
\label{prop:defocusing R^d}Fix a number $E_{0}>0$. Let $u_{0}\in\dot{H}^{1}(\R^{d}),d\ge3$
be initial data such that $E(u_{0})\le E_{0}$. There exists a unique
Duhamel solution $u\in C^{0}\dot{H}^{1}\cap L_{t,x}^{\frac{2d+4}{d-2}}(\R\times\R^{d})$
to (\ref{eq:NLS R^d}) with the defocusing sign, which is global and
scatters both forward and backward in time. Moreover, we have
\[
\norm u_{L_{t,x}^{\frac{2d+4}{d-2}}}\les_{E_{0}}1.
\]
\end{prop}

\begin{proof}
This is shown in \cite{colliander2008global,ryckman2007global,visan2007defocusing}.
\end{proof}
A similar result for the focusing case is known for dimensions $d\ge4$:
\begin{prop}[\cite{killip2010focusing,dodson2019global}]
\label{prop:focusing R^d}Let $u_{0}\in\dot{H}^{1}(\R^{d}),d\ge4$
be initial data such that
\[
\norm{\na u_{0}}_{L^{2}}\le K<\norm{\na W}_{L^{2}}
\]
and
\[
E(u_{0})<E(W).
\]
Then, there exists a unique Duhamel solution $u\in C^{0}\dot{H}^{1}\cap L_{t,x}^{\frac{2d+4}{d-2}}(\R\times\R^{d})$
to (\ref{eq:NLS R^d}) with the focusing sign, which is global and
scatters both forward and backward in time. Moreover, we have
\[
\norm u_{L_{t,x}^{\frac{2d+4}{d-2}}}\les_{K}1.
\]
\end{prop}

Here, we denote by $W$ the ground state of (\ref{eq:NLS R^d}) on
$\R^{d}$, explicitly given in (\ref{eq:W}).

\subsection{Energy trapping of energy-critical focusing NLS on $\protect\T^{d}$}

Yue \cite{YUE2021754} used a variational structure to introduce modified
energies, which give sufficient conditions for energy trappings on
$\T^{d}$. While \cite{YUE2021754} only dealt four-dimensional cases,
the energy trapping works on general torus $\T^{d},d\ge3$, hence
we state the result on general $\T^{d}$.

We denote by $c_{*}$ and $C_{d}$ the constants given in \cite[Lemma 2.1]{YUE2021754},
which depend only on the domain $\T^{d}$. Explicitly, we define $C_{d}:=\norm W_{L^{2+a}}/\norm{\na W}_{L^{2}}$
and choose $c_{*}$ as a positive constant such that for $u\in H^{1}(\T^{d})$,
\[
\norm u_{L^{2+a}(\T^{d})}^{2}\le C_{d}^{2}\left(\norm{\na u}_{L^{2}(\T^{d})}^{2}+c_{*}\norm u_{L^{2}(\T^{d})}^{2}\right).
\]
We define $c_{0}$ as the infimum of all such available $c_{*}$.

For $u\in H^{1}(\T^{d})$ and $c>0$, we define the modified energy
$E_{c}(u)$ as 

\[
E_{c}(u)=E(u)+\frac{c}{2}\norm u_{L^{2}}^{2},
\]
which is conserved by the mass and energy conservation laws.
\begin{rem}
Although the condition of Theorem \ref{thm:focusing GWP} contains
$c_{0}$, the inequality conditions are strict and thus $c_{0}$ can
be replaced by some $c_{*}>c_{0}$.
\end{rem}

We recall the following energy trapping principle:
\begin{prop}[{Energy trapping, \cite[Theorem 2.5]{YUE2021754}}]
\label{prop:energy trapping}Let $u_{0}\in H^{1}(\T^{d}),d\ge3$
be such that for some $\delta_{0}>0$ and $c_{*}>c_{0}$,
\[
\norm{\na u_{0}}_{L^{2}}^{2}+c_{*}\norm{u_{0}}_{L^{2}}^{2}<\norm{\na W}_{L^{2}}^{2}
\]
and
\[
E_{c_{*}}(u_{0})<(1-\delta_{0})E(W).
\]
Let $u$ be a maximum lifespan solution to (\ref{eq:NLS T^d}) with
the initial data $u_{0}$ and the lifespan $I\ni0$. There exists
$\bar{\delta}=\bar{\delta}(\delta_{0})>0$ such that for all $t\in I$,
\[
\norm{\na u(t)}_{L^{2}}^{2}<(1-\bar{\delta})\norm{\na W}_{L^{2}}^{2}.
\]
\end{prop}

\section{Frames and cutoff solutions\label{sec:Basic-properties-of}}

\subsection{Periodic extensions and frames}

The proofs of the main results capture limit profiles from concentrating
bubbles by taking weak limits of their rescales. The limit profiles
lie on $\R\times\R^{d}$; hence, it is convenient to first extend
a function on $\R\times\T^{d}$ to that on $\R\times\R^{d}$ and then
rescale the coordinates. How the function is extended is not essential
to our roadmap; one could, e.g., use charts and bump functions for
general domains. In our setting, we extend a function on $\R\times\T^{d}$
to a spatially $2\pi\Z^{d}$-periodic function on $\R\times\R^{d}$
and rescale it, as explained below.

We define the mappings of spacetime functions. We denote a triple
$\left(N_{*},t_{*},x_{*}\right)\in2^{\N}\times\R\times\T^{d}$ of
scale, time, and space parameters. For $f\in L_{\loc}^{1}(\R\times\T^{d})$
and $\left(N_{*},t_{*},x_{*}\right)\in2^{\N}\times\R\times\T^{d}$,
we denote by $\iota_{(N_{*},t_{*},x_{*})}f$ the $L^{\infty}\dot{H}^{1}$-critically
rescaled function
\[
\iota_{(N_{*},t_{*},x_{*})}f(t,x):=N_{*}^{1-\frac{d}{2}}f\left(N_{*}^{-2}t+t_{*},N_{*}^{-1}x+x_{*}+2\pi\Z^{d}\right)\in L_{\loc}^{1}(\R\times\R^{d}).
\]
We also use the $L^{\infty}\dot{H}^{-1}$-critical rescale $\iota'_{(N_{*},t_{*},x_{*})}=N_{*}^{-2}\iota_{(N_{*},t_{*},x_{*})}$.
\begin{defn}
\label{def:frame}A sequence of triples of parameters $\left\{ \left(N_{n},t_{n},x_{n}\right)\right\} _{n\in\N}$
in $2^{\N}\times\R\times\T^{d}$ is said to be a \emph{frame} if $\lim_{n\rightarrow\infty}N_{n}=\infty$.
\end{defn}

We frequently work on distributional weak limits of $\iota_{\mc O_{n}}f_{n}$
for a sequence of functions $f_{n}:\R\times\T^{d}\rightarrow\C$.
\begin{defn}
Let $q,r\in[1,\infty]$ and $k\in\Z$. A sequence of functions $\left\{ f_{n}\right\} $
in $L_{\loc}^{1}(\R\times\R^{d})$ is said to be \emph{uniformly locally
bounded }in $L^{q}\dot{W}^{k,r}$ if
\[
\sup_{R\in2^{\N}}\limsup_{n\rightarrow\infty}\norm{f_{n}}_{L^{q}\dot{W}^{k,r}\left(\R\times B_{R}(0)\right)}<\infty,
\]
where $B_{R}(0):=\left\{ x\in\R^{d}:|x|\le R\right\} $.
\end{defn}

For instance, for any frame $\left\{ \mc O_{n}\right\} $ and any
bounded sequence of functions $\left\{ f_{n}\right\} $ in $L^{q}W^{k,r}(\R\times\T^{d})$,
if $L^{q}\dot{W}^{k,r}$ and $L^{\infty}\dot{H}^{1}$ have the same
scaling, since the support of $\iota_{\mc O_{n}}(\chi_{\R\times[-\pi,\pi)^{d}})$
grows to $\R\times\R^{d}$ as $N_{n}\rightarrow\infty$, $\left\{ \iota_{\mc O_{n}}f_{n}\right\} $
is uniformly locally bounded in $L^{q}\dot{W}^{k,r}$.

\subsection{Cutoff solutions}

In this work, to consider a solution $u$ locally within a short time
interval $I\subset\R$, it is often convenient to consider an extension
of $u\mid_{I}$ by linear evolutions on $\R\setminus I$. For this,
we introduce the concept of a cutoff solution.

A pair $(u,I)$, where $u\in C^{0}\mc S'\cap L_{\loc}^{1+a}(\R\times\R^{d})$
and $I\subset\R$ is an interval, is a \emph{cutoff solution }to (\ref{eq:NLS R^d})
if $u$ is a Duhamel solution to
\[
iu_{t}+\De u=\chi_{I}\cdot\mc N(u).
\]
Here, $I$ is possibly empty or $\R$ itself (i.e., both linear evolutions
and purely nonlinear solutions are cutoff solutions). A cutoff solution
on $\T^{d}$ is defined similarly. Equivalently, $u$ can be regarded
as the continuous extension of $u\mid_{I}$ by linear evolutions.

The following lemma enables one to regard a weak limit of periodic
extensions $\iota_{\mc O_{n}}$ of cutoff solutions on $\T^{d}$ as
a strong (cutoff) solution on $\R^{d}$.
\begin{lem}
\label{lem:weak lim is sol}Let $\left\{ (v_{n},I_{n})\right\} $
be a sequence of cutoff solutions to (\ref{eq:NLS R^d}) uniformly
locally bounded in $C^{0}\dot{H}^{1}\cap L_{t,x}^{\frac{2d+4}{d-2}}(\R\times\R^{d})$.
Passing to a subsequence, $v_{n}$ converges weakly and almost everywhere
to a cutoff solution $(v_{*},I_{*})$, $v_{*}\in C^{0}\dot{H}^{1}\cap L_{t,x}^{\frac{2d+4}{d-2}}(\R\times\R^{d})$
to (\ref{eq:NLS R^d}), which scatters both forward and backward in
time. Furthermore, for every $T\in\R$, we have $v_{n}(T)\weak v_{*}(T)$.
\end{lem}

\begin{proof}
Since $i\d_{t}v_{n}=-\De v_{n}+\mc N(v_{n})$ is bounded in $C^{0}H_{\loc}^{-1}$,
\begin{equation}
\left\{ v_{n}\right\} \text{ is bounded in }C^{0}H_{\loc}^{1}\cap\dot{W}^{1,\infty}H_{\loc}^{-1}\hook C^{0}H_{\loc}^{1}\cap C^{0,1/2}L_{\loc}^{2}.\label{eq:sup vn}
\end{equation}
Thus, passing to a subsequence, $v_{n}$ converges almost everywhere;
let $v_{*}$ be the pointwise limit. By the uniform local boundedness
of $\left\{ v_{n}\right\} $ in $C^{0}\dot{H}^{1}\cap L_{t,x}^{\frac{2d+4}{d-2}}$
and (\ref{eq:sup vn}), we have
\begin{equation}
v_{n}\weak v_{*}\in L^{\infty}\dot{H}^{1}\cap L_{t,x}^{\frac{2d+4}{d-2}}\cap C^{0,1/2}L_{\loc}^{2}.\label{eq:v*}
\end{equation}
By (\ref{eq:sup vn}), $\left\{ v_{n}\right\} $ is equicontinuous
in $L_{\loc}^{2}(\R^{d})$, thus $v_{n}(T)\weak v_{*}(T)$ further
holds for every $T\in\R$.

It remains to show that $v_{*}$ is a cutoff solution. We proceed
conventionally by first showing that $v_{*}$ is a cutoff solution
in the distributional sense, then using a mollification argument to
show that such $v_{*}$ is also a solution in the Duhamel sense.

Firstly, we show that $v_{*}$ is a distributional cutoff solution.
Passing to a subsequence, there exists an interval $I_{*}$ such that
$\chi_{I_{n}}$ converges almost everywhere to $\chi_{I_{*}}$. Since
$v_{n}$ is bounded in $L_{\loc}^{\frac{2d+4}{d-2}}(\R\times\R^{d})$
and converges almost everywhere to $v_{*}$, $\mc N(v_{n})$ is bounded
in $L_{\loc}^{2}(\R\times\R^{d})$ and converges almost everywhere
to $\mc N(v_{*})$. Thus, taking weak limits of both sides of (\ref{eq:NLS R^d})
gives
\begin{equation}
(i\d_{t}+\De)v_{*}=\chi_{I_{*}}\cdot\mc N(v_{*}).\label{eq:dist}
\end{equation}
We check that $v_{*}$ is a Duhamel cutoff solution. We use a mollification
argument. Let $\phi\in C_{0}^{\infty}(\R\times\R^{d})$ be a function
such that $\int_{\R\times\R^{d}}\phi dxdt=1$. For $n\in\N$, let
$\phi_{n}(t,x):=n^{d+1}\phi(nt,nx)$. Taking a convolution with $\phi_{n}$
to (\ref{eq:dist}), we have
\[
(i\d_{t}+\De)v_{*}*\phi_{n}=\left(\chi_{I_{*}}\cdot\mc N(v_{*})\right)*\phi_{n},
\]
which can be rewritten in the Duhamel form
\begin{equation}
\left(v_{*}*\phi_{n}\right)(t)=e^{it\De}\left(v_{*}*\phi_{n}\right)(0)-i\int_{0}^{t}e^{i(t-s)\De}\left(\left(\chi_{I}\cdot\mc N(v_{*})\right)*\phi_{n}\right)(s)ds.\label{eq:weak Duhamel}
\end{equation}
Since $v_{*}\in C^{0}L_{\loc}^{2}$, we have $\left(v_{*}*\phi_{n}\right)(0)\weak v_{*}(0)$.
Since $v_{*}\in L_{t,x}^{\frac{2d+4}{d-2}}(\R\times\R^{d})$, we have
$\chi_{I_{*}}\cdot\mc N(v_{*})\in L_{t,x}^{2}(\R\times\R^{d})$. Thus,
taking a weak limit by $n\rightarrow\infty$ of (\ref{eq:weak Duhamel})
yields
\[
v_{*}=e^{it\De}v_{*}(0)-i\int_{0}^{t}e^{i(t-s)\De}\left(\chi_{I}\cdot\mc N(v_{*})\right)(s)ds,
\]
which implies that $(v_{*},I_{*})$ is a Duhamel cutoff solution.
By Lemma \ref{lem:local theory}, $v_{*}\in C^{0}\dot{H}^{1}$ further
holds, hence finishing the proof.
\end{proof}

\section{A solution space and an inverse embedding property\label{sec:Solution-spaces-and}}

In this section, we define and investigate a solution space $Z^{s}=Z_{\s}^{s}$
modifying that of the local well-posedness result in \cite{LWP(working)}.
We bring parameters $\s$ and $p$ given in \cite{LWP(working)},
which depend only on $d$. As in \cite[Section 3]{LWP(working)},
we fix the parameters $\s$ and $p$ as follows:

\begin{equation}
\s=\s(d)\ll1\text{ and }p=\frac{d+2}{\frac{d}{2}-\s}.\label{eq:=00005Cs,p def}
\end{equation}
Let $\psi:\R\rightarrow[0,\infty)$ be a smooth even bump function
such that $\psi|_{[-1,1]}\equiv1$ and $\text{supp}(\psi)\subset[-\frac{11}{10},\frac{11}{10}]$.
\begin{defn}
\label{def:Z^s_theta def}For $d\in\N$, $s\in\R$, and $\theta\in[0,\s]$,
we define $Z_{\theta}^{s}=Z_{\theta}^{s}(\R\times\T^{d})$ as a Banach
space given by the norm
\begin{align}
\norm u_{Z_{\theta}^{s}} & =\max_{q\in\text{ \ensuremath{\left[p,\frac{1}{\s}\right]} }}\norm{N^{s-\s}\norm{\psi u_{N}}_{L^{q}L^{r}}}_{\ell^{2}(N\in2^{\N})}\label{eq:Z^s def}\\
 & +\max_{\al\in\left[\s,\frac{1}{p}-\s\right]}\norm{N^{s}\max_{\substack{R\in2^{\N}\\
R\le8N
}
}\left(\text{\ensuremath{\jp{N/R}^{-\theta}}}+R^{-2\theta}\right)R^{-\be}\norm{\norm{\psi P_{\le8R}I_{Rk}u_{N}}_{B_{p,1}^{\al}L^{p}}}_{\ell^{2}(k\in\Z^{d})}}_{\ell^{2}(N\in2^{\N})},\nonumber 
\end{align}
where we denote $\jp x:=\sqrt{|x|^{2}+1}$ for $x\in\R$ and the scaling
conditions $\frac{d}{r}:=\frac{d}{2}-\s-\frac{2}{q}$ and $\be:=\s+2\al$
are imposed.
\end{defn}

In \cite{LWP(working)}, the $Z_{0}^{s}$ space (originally denoted
by $Z^{s}$ in \cite{LWP(working)}) was first defined and used as
a solution norm. The $Z_{0}^{s}$ norm was initially introduced with
sharp Littlewood-Paley cutoffs, but they work the same with smooth
cutoffs, which is our setting.

In this paper, $Z_{\s}^{s}$ is mainly used. Since $Z_{\theta}^{s}$
is weaker with higher $\theta$, adopting $Z_{\theta}^{s}$ instead
of $Z_{0}^{s}$ means that we use a weaker solution norm for bootstrapping.
This is beneficial for the inverse property of the embedding $Y^{s}\hook Z_{\theta}^{s}$
(Proposition \ref{prop:inverse embedding}). Indeed, the factor $\theta>0$
is crucial; for instance, a linear combination of Galilean packets
$I_{Rk}P_{\le R}e^{it\De}\delta$, $k\in\Z^{d}$ satisfies $\norm u_{Y^{s}}\sim\norm u_{Z_{0}^{s}}$.

Although $Z_{\s}^{s}$ is weaker than $Z_{0}^{s}$ previously used
in \cite{LWP(working)}, all estimates given in \cite{LWP(working)}
stay true with $Z_{\s}^{s}$. The second line of (\ref{eq:Z^s def}),
where the decay exponent $\theta$ appears, was used in \cite{LWP(working)}
only for the bilinear Strichartz estimate \cite[(3.25)]{LWP(working)}.
That estimate stays true, as explained in the proof of Lemma \ref{lem:Nu bound}.

We collect elementary properties of $Z_{\theta}^{s}$ spaces.
\begin{prop}[{\cite[Lemma 3.6]{LWP(working)}}]
\label{lem:summary of embeds}Let $s\in\R$ and $\theta\in[0,\s]$.
We have the following properties:
\begin{itemize}
\item We have the embedding
\begin{equation}
\ell_{s}^{2}(Z_{\theta}^{0})'=(Z_{\theta}^{-s})'\hook(Y^{-s})'\stackrel{K^{+}}{\longrightarrow}Y^{s}\hook Z_{\theta}^{s}=\ell_{s}^{2}Z_{\theta}^{0}.\label{eq:Z^s-Y^s}
\end{equation}
\item For a finite interval $I\subset\R$, we have
\begin{equation}
\norm{u\cdot\chi_{I}}_{Z_{\theta}^{s}}\les\norm u_{Z_{\theta}^{s}}.\label{eq:step Z^s}
\end{equation}
This estimate is uniform on the choice of $I$.
\item Let $\al\in[\s,\frac{1}{p}-\s]$ and $\be=\s+2\al$. We have
\begin{equation}
\norm{\psi u}_{\ell_{s-\be}^{2}B_{p,1}^{\al}L^{p}}\les\norm u_{Z_{\theta}^{s}}.\label{eq:time Besov embed Z^s}
\end{equation}
\end{itemize}
\end{prop}

\begin{proof}
These are shown in \cite[Lemma 3.6]{LWP(working)} for the case $\theta=0$.
The proofs are identical for general $\theta\ge0$.
\end{proof}
Indeed, the $\s$-gap for the interval $\al\in\left[\s,\frac{1}{p}-\s\right]$
in (\ref{eq:Z^s def}) is artificial and can be relaxed to any compact
subinterval of $(0,\frac{1}{p})$. In particular, we have the following:
\begin{lem}[{\cite[Lemma 3.6]{LWP(working)}}]
Let $s\in\R$. For $\al\in(0,\frac{1}{p})$, we have
\begin{equation}
\norm{N^{s}\max_{\substack{R\in2^{\N}\\
R\le8N
}
}R^{-\be}\norm{\norm{\psi P_{\le8R}I_{Rk}u_{N}}_{B_{p,1}^{\al}L^{p}}}_{\ell^{2}(k\in\Z^{d})}}_{\ell^{2}(N\in2^{\N})}\les_{\al}\norm u_{Y^{s}},\label{eq:Z^s-Y^s relax}
\end{equation}
where the scaling condition $\be=\s+2\al$ is imposed.
\end{lem}

\begin{proof}
This is just a part of (\ref{eq:Z^s-Y^s}), bringing the second line
of (\ref{eq:Z^s def}) with $\al$ replaced by a member of $(0,\frac{1}{p})$.
The proof in \cite[Lemma 3.6]{LWP(working)} used only that $0<\al<\frac{1}{p}$
and works identically here.
\end{proof}
\begin{lem}[{\cite[(4.8)]{LWP(working)}}]
\label{lem:Nu bound}Let $\theta\in[0,\s]$. For $u\in Z_{\theta}^{1}$
such that $\supp(u)\subset[-1,1]\times\T^{d}$, we have
\begin{equation}
\norm{\mc N(u)}_{(Z_{\theta}^{-1})'}\les\norm u_{Z_{\theta}^{1}}^{1+a}.\label{eq:Nu bound}
\end{equation}
\end{lem}

\begin{proof}
(\ref{eq:Nu bound}) is \cite[(4.8)]{LWP(working)} for the case $\theta=0$.
In \cite{LWP(working)}, the key bilinear Strichartz estimate \cite[(3.25)]{LWP(working)}
is the only place that distinguishes $Z_{\theta}^{1}$ and $Z_{0}^{1}$.
\cite[(3.25)]{LWP(working)} contains a bilinear gain by $\text{\ensuremath{\jp{N/R}^{-\s_{1}}}}+R^{-2\s_{1}}$,
which is the first-appearing place of the exponent $\s_{1}\gg\s$
and allows a perturbation by $\tilde{\s_{1}}=\s_{1}-\s$. Thus, the
proof of \cite[(4.8)]{LWP(working)} works identically on showing
(\ref{eq:Nu bound}) for general $\theta\in[0,\s]$.
\end{proof}
Just as commented after Definition \ref{def:Z^s_theta def}, we mostly
use $Z_{\s}^{s}$ spaces in later proofs. For simplicity of notation,
we denote $Z^{s}:=Z_{\s}^{s}$.

We show an inverse property of the embedding $Y^{1}\hook Z^{1}=Z_{\s}^{1}$
localized to one dyadic frequency support, which is the main reason
why we adopted the space $Z_{\s}^{1}$. This can be naturally expected
since $Y^{1}$ is a space of regularity $1$, while $Z^{1}$ is based
on Besov spaces of regularities $1-\s<1$ (up to a fractional exchange
of $\d_{t}$ and $\De$).
\begin{defn}
A family of distributions $\left\{ f_{n}\right\} $ on either $\R^{d}$
or $\R\times\R^{d}$ is said to be \emph{weakly nonzero} if for any
subsequence $\left\{ n_{k}\right\} $, $f_{n_{k}}$ does not converge
weakly to zero.
\end{defn}

\begin{prop}
\label{prop:inverse embedding}Let $C<\infty$ and $\epsilon>0$.
Let $N_{n}\rightarrow\infty$ be a sequence of dyadic numbers. Let
$I_{n}\subset[0,1]$ be a sequence of intervals such that $\left|I_{n}\right|\rightarrow0$.
Let $\left\{ u_{n}\right\} $ be a sequence in $Y^{1}$ such that
$\supp(u_{n})\subset I_{n}\times\T^{d}$,
\begin{equation}
\sup_{n}\norm{P_{N_{n}}u_{n}}_{Y^{1}}\le C,\label{eq:inv sup}
\end{equation}
and
\begin{equation}
\inf_{n}\norm{P_{N_{n}}u_{n}}_{Z^{1}}\ge\epsilon.\label{eq:inv inf}
\end{equation}
Then, there exists a frame $\left\{ \mc O_{n}\right\} =\left\{ (N_{n},t_{n},x_{n})\right\} $,
where $t_{n}\in I_{n}$ and $x_{n}\in\T^{d}$, such that $\left\{ \iota_{\mc O_{n}}u_{n}(0)\right\} $
is weakly nonzero.
\end{prop}

\begin{proof}
Up to time-translations, we may assume $I_{n}=[0,T_{n})$, $T_{n}\rightarrow0$.
Recalling (\ref{eq:Z^s def}), for each $n$, either of the following
holds:
\begin{enumerate}
\item There exists $q_{n}\in[p,\frac{1}{\s}]$ such that for $\frac{d}{r_{n}}:=\frac{d}{2}-\s-\frac{2}{q_{n}}$,
\begin{equation}
\norm{\psi P_{N_{n}}u_{n}}_{L^{q_{n}}L^{r_{n}}}\gtrsim N_{n}^{\s-1}.\label{eq:(1)}
\end{equation}
\item There exist $\al_{n}\in[\s,\frac{1}{p}-\s]$ and a dyadic number $R_{n}\le8N$
such that for $\be_{n}:=\s+2\al_{n}$,
\begin{equation}
\left(\text{\ensuremath{\jp{N_{n}/R_{n}}^{-\s}}}+R_{n}^{-2\s}\right)R_{n}^{-\be_{n}}\norm{\norm{\psi P_{\le8R_{n}}I_{R_{n}k}P_{N_{n}}u_{n}}_{B_{p,1}^{\al_{n}}L^{p}}}_{\ell^{2}(k\in\Z^{d})}\gtrsim N_{n}^{-1}.\label{eq:(2)}
\end{equation}
\end{enumerate}
\emph{Case 1.} By interpolation, we assume $q_{n}\in\left\{ p,\frac{1}{\s}\right\} $.
Passing to a subsequence, we assume $q_{n}=q\in\left\{ p,\frac{1}{\s}\right\} $.
By (\ref{eq:inv sup}), we have
\begin{equation}
\norm{\psi P_{N_{n}}u_{n}}_{L^{\infty}L^{2_{\s}}}\les N_{n}^{\s}\norm{\psi P_{N_{n}}u_{n}}_{L^{\infty}L^{2}}\les N_{n}^{\s-1}\norm{P_{N_{n}}u_{n}}_{Y^{1}}\les N_{n}^{\s-1}.\label{eq:1-2}
\end{equation}
Since $p\le q<\infty$, extrapolating between (\ref{eq:(1)}) and
(\ref{eq:1-2}) gives
\begin{equation}
\norm{\psi P_{N_{n}}u_{n}}_{L_{t,x}^{p}}\gtrsim N_{n}^{\s-1}.\label{eq:1-11}
\end{equation}
Let $\tilde{\s}\in(0,\s)$ and $\tilde p:=\frac{d+2}{\frac{d}{2}-\tilde{\s}}$.
By (\ref{eq:cube Strichartz}) and (\ref{eq:inv sup}), we have
\begin{equation}
\norm{\psi P_{N_{n}}u_{n}}_{L_{t,x}^{\tilde p}}\les N^{\tilde{\s}}\norm{P_{N_{n}}u_{n}}_{Y^{0}}\les N_{n}^{\tilde{\s}-1}.\label{eq:1-12}
\end{equation}
Since $\tilde p<p<\infty$, extrapolating between (\ref{eq:1-11})
and (\ref{eq:1-12}) yields
\[
\norm{\psi P_{N_{n}}u_{n}}_{L_{t,x}^{\infty}}\gtrsim N_{n}^{\frac{d}{2}-1}.
\]
Thus, there exists a point $(t_{n},x_{n})\in I_{n}\times\T^{d}$ such
that 
\[
\left|\jp{u_{n}(t_{n}),\delta_{N_{n}}(\cdot-x_{n})}_{L^{2}(\T^{d})}\right|=\left|P_{N_{n}}u_{n}(t_{n},x_{n})\right|\gtrsim N_{n}^{\frac{d}{2}-1}.
\]
We choose the frame $\left\{ \mc O_{n}\right\} $ as $\left\{ \mc O_{n}\right\} =\left\{ (N_{n},t_{n},x_{n})\right\} $.
The sequence $\jp{\iota_{\mc O_{n}}u_{n}(0),\delta_{1}}_{L^{2}(\R^{d})}=\jp{N_{n}^{1-\frac{d}{2}}u_{n}(t_{n}),\delta_{N_{n}}(\cdot-x_{n})}_{L^{2}(\T^{d})}$
does not converge to zero on any subsequence, hence we are done for
this case.

\emph{Case 2.} By interpolation, we assume $\al_{n}\in\left\{ \s,\frac{1}{p}-\s\right\} $.
Passing to a subsequence, we assume $\al_{n}=\al\in\left\{ \s,\frac{1}{p}-\s\right\} $.
Fix $\tilde{\al}\in(\al,\frac{1}{p})$. Denote $\be:=\s+2\al$ and
$\tilde{\be}:=\s+2\tilde{\al}$. By (\ref{eq:Z^s-Y^s relax}), we
have
\begin{equation}
R_{n}^{-\tilde{\be}}\norm{\norm{\psi P_{\le8R_{n}}I_{R_{n}k}P_{N_{n}}u_{n}}_{B_{p,1}^{\tilde{\al}}L^{p}}}_{\ell^{2}(k\in\Z^{d})}\les N_{n}^{-1}\norm{P_{N_{n}}u_{n}}_{Y^{1}}\les N_{n}^{-1}.\label{eq:(2)R}
\end{equation}

By (\ref{eq:inv sup}) and (\ref{eq:Z^s-Y^s}), we have $\norm{u_{n}}_{Z_{0}^{1}}\les1$,
comparing which and (\ref{eq:(2)}), the decay part $\text{\ensuremath{\jp{N_{n}/R_{n}}^{-\s}}}+R_{n}^{-2\s}$
of (\ref{eq:(2)}) cannot decrease to zero. Hence we may assume either
$R_{n}\sim1$ or $R_{n}\sim N_{n}$.

\emph{Case 2-1: $R_{n}\sim1$. }By (\ref{eq:(2)}) and (\ref{eq:(2)R}),
we have
\[
\norm{\norm{\psi P_{\le8R_{n}}I_{R_{n}k}P_{N_{n}}u_{n}}_{B_{p,1}^{\tilde{\al}}L^{p}}}_{\ell^{2}(k\in\Z^{d})}\les\norm{\norm{\psi P_{\le8R_{n}}I_{R_{n}k}P_{N_{n}}u_{n}}_{B_{p,1}^{\al}L^{p}}}_{\ell^{2}(k\in\Z^{d})},
\]
which cannot happen since $\tilde{\al}>\al$ and $\supp\left(\psi P_{\le8R_{n}}I_{R_{n}k}P_{N_{n}}u_{n}\right)\subset I_{n}\times\T^{d}$
holds with $\left|I_{n}\right|\rightarrow0$ for each $k\in\Z^{d}$.
Thus, the subcase $R_{n}\sim1$ is vacuous.

\emph{Case 2-2: }$R_{n}\sim N_{n}$. Since $I_{R_{n}k}$ translates
the spatial Fourier support by $R_{n}k$, $P_{\le8R_{n}}I_{R_{n}k}P_{N_{n}}$
is zero for $\left|k\right|\gg1$. Thus, only $O(1)$ number of integer
points $k\in\Z^{d}$ contribute to (\ref{eq:(2)}) and there exists
$k_{n}\in\Z^{d}$ such that
\begin{equation}
R_{n}^{-\be}\norm{\psi P_{\le8R_{n}}I_{R_{n}k_{n}}P_{N_{n}}u_{n}}_{B_{p,1}^{\al}L^{p}}\gtrsim N_{n}^{-1}.\label{eq:(2)'}
\end{equation}
By (\ref{eq:Besov real interp E s0s1}), an extrapolation between
(\ref{eq:(2)R}) and (\ref{eq:(2)'}) yields
\begin{equation}
R_{n}^{-\s}\norm{\psi P_{\le8R_{n}}I_{R_{n}k_{n}}P_{N_{n}}u_{n}}_{B_{p,\infty}^{0}L^{p}}\gtrsim N_{n}^{-1}.\label{eq:2-2}
\end{equation}
By (\ref{eq:2-2}), the embedding $L_{t,x}^{p}\hook B_{p,\infty}^{0}L^{p}$,
and the $L_{t,x}^{p}$-norm conservation of the Galilean transform
$I_{R_{n}k_{n}}$, we have
\begin{equation}
\norm{\psi P_{N_{n}}u_{n}}_{L_{t,x}^{p}}\gtrsim R_{n}^{\s}N_{n}^{-1}\sim N_{n}^{\s-1},\label{eq:2-3}
\end{equation}
which is just (\ref{eq:1-11}) we obtained in the proof for Case 1;
proceeding as earlier finishes the proof. 
\end{proof}
\begin{rem}
Proposition \ref{prop:inverse embedding} concerns only functions
localized in one dyadic Fourier support. This is the best one can
expect since $Z^{1}=\ell^{2}Z^{1}$ and $Y^{1}=\ell^{2}Y^{1}$. This
is in contrast to previous results \cite{ionescu2012energy,YUE2021754}
for $d=3,4$, where an inverse property for the inequality $\norm{e^{it\De}\phi}_{L^{\infty}B_{\infty,\infty}^{-d/2}(I\times\T^{d})}\les\norm{\phi}_{L^{2}(\T^{d})}$
was enough.

Thus, the concentration of a solution to (\ref{eq:NLS T^d}) is not
directly obvious from Proposition \ref{prop:inverse embedding}. This
is a major step of this work; we continue the discussion in Section
\ref{subsec:nonlinear conc}.
\end{rem}

\section{Global well-posedness of NLS on $\protect\T^{d}$\label{sec:Global-well-posedness-of NLS}}

\subsection{Weak scattering behaviors\label{subsec:Weak-scattering-behaviors}}

In this subsection, we show uniform convergences of scattering limits
over sequences of solutions.

For $\R^{d}$, we show that for a bounded sequence of solutions to
(\ref{eq:NLS R^d}), weak convergence to the scattering limit as $t\rightarrow-\infty$
uniformly occurs. A similar property is shown for $\T^{d}$ with respect
to any frame $\left\{ \mc O_{n}\right\} $. This phenomenon is generic;
the only ingredient is any decay of a linear Schr\"odinger propagation
(from $\phi\in C_{0}^{\infty}(\R^{d})$ or $P_{N}\delta_{\T^{d}}$)
in a solution norm, which implies, in view of duality, that the inhomogeneous
evolution from time $t$ too far from $0$ is almost negligible to
the scattering limit in the weak sense.
\begin{lem}
\label{lem:scatter R^d}Let $d\ge3$. Let $\left\{ (v_{n},I_{n})\right\} $
be a bounded sequence of cutoff solutions to (\ref{eq:NLS R^d}) in
$C^{0}\dot{H}^{1}\cap L_{t,x}^{\frac{2d+4}{d-2}}(\R\times\R^{d})$.
Then, 
\[
\left\{ e^{iT\De}v_{n}(-T)\right\} _{n\in\N}
\]
is uniformly convergent (in the weak sense) as $T\rightarrow\infty$.
\end{lem}

\begin{proof}
Explicitly, we show that for every $\phi\in C_{0}^{\infty}(\R^{d})$,
\begin{equation}
\limsup_{T_{1},T_{2}\rightarrow\infty}\sup_{n\in\N}\left|\jp{\phi,e^{iT_{1}\De}v_{n}(-T_{1})-e^{iT_{2}\De}v_{n}(-T_{2})}_{L^{2}(\R^{d})}\right|=0.\label{eq:weak scatter R^d claim}
\end{equation}
Applying Lemma \ref{lem:local theory} to $v_{n}\mid_{I_{n}}$ yields
\begin{equation}
\sup_{n\in\N}\norm{v_{n}}_{L^{2}\dot{W}^{1,2_{1}}(\R\times\R^{d})}<\infty.\label{eq:sup W1}
\end{equation}
For $T_{1},T_{2}>0$ and $n\in\N$, we have
\begin{align*}
\left|\jp{\phi,e^{iT_{1}\De}v_{n}(-T_{1})-e^{iT_{2}\De}v_{n}(-T_{2})}_{L^{2}}\right| & =\left|\jp{\phi,\int_{[T_{1},T_{2}]\cap I_{n}}e^{is\De}\mc N(v_{n}(-s))ds}_{L^{2}}\right|\\
 & =\left|\int_{[T_{1},T_{2}]\cap I_{n}}\jp{e^{-is\De}\phi,\mc N(v_{n}(-s))}_{L^{2}}ds\right|,
\end{align*}
then by (\ref{eq:sup W1}) we continue to estimate
\begin{align}
 & \les\norm{\chi_{[-T_{1},-T_{2}]}e^{it\De}\phi}_{L^{2}\dot{W}^{-1,2_{1}}}\norm{\mc N(v_{n})}_{L^{2}\dot{W}^{1,2_{-1}}}\nonumber \\
 & \les\norm{\chi_{[-T_{1},-T_{2}]}e^{it\De}\phi}_{L^{2}\dot{W}^{-1,2_{1}}}.\label{eq:T1T2}
\end{align}
Taking the limit $T_{1},T_{2}\rightarrow\infty$ to (\ref{eq:T1T2})
yields
\begin{align*}
 & \limsup_{T_{1},T_{2}\rightarrow\infty}\sup_{n\in\N}\left|\jp{\phi,e^{iT_{1}\De}v_{n}(-T_{1})-e^{iT_{2}\De}v_{n}(-T_{2})}_{L^{2}}\right|\\
 & \les\limsup_{T_{1},T_{2}\rightarrow\infty}\norm{\chi_{[-T_{1},-T_{2}]}e^{it\De}\phi}_{L^{2}\dot{W}^{-1,2_{1}}},
\end{align*}
which is $0$ since $e^{it\De}\phi$ lies in the Strichartz space
$L^{2}\dot{W}^{-1,2_{1}}$, finishing the proof.
\end{proof}
We show a similar uniform weak scattering property on $\T^{d}$. While
we used only the boundedness of a Strichartz norm of $e^{it\De}\phi$
for $\R^{d}$, since the scaling size matters on $\T^{d}$, we employ
a decay estimate of the kernel $e^{it\De}\delta_{N}=e^{it\De}P_{N}\delta$
uniform on $N$, explicitly the following preparatory fact known as
the extinction lemma:
\begin{lem}[Extinction lemma, \cite{ionescu2012energy}]
\label{prop:kernel 0}Let $d\ge3$. We have
\begin{equation}
\limsup_{\substack{\epsilon\rightarrow0\\
T\rightarrow\infty
}
}\sup_{\substack{N\in2^{\N}\\
TN^{-2}<\epsilon
}
}N^{1-\frac{d}{2}}\norm{\chi_{[TN^{-2},\epsilon]}e^{it\De}\delta_{N}}_{Z^{-1}}=0.\label{eq:kernel 0}
\end{equation}
\end{lem}

\begin{proof}
By the definition of $Z^{s}$, (\ref{eq:kernel 0}) can be rewritten
as
\[
\limsup_{\substack{\epsilon\rightarrow0\\
T\rightarrow\infty
}
}\sup_{\substack{N\in2^{\N}\\
TN^{-2}<\epsilon
}
}N^{-1-\frac{d}{2}}\norm{\chi_{[TN^{-2},\epsilon]}e^{it\De}\delta_{N}}_{Z^{1}}=0.
\]
Since the length of the interval $[TN^{-2},\epsilon]$ vanishes to
zero, by Proposition \ref{prop:inverse embedding}, it suffices to
show
\begin{equation}
\limsup_{\substack{\epsilon\rightarrow0\\
T\rightarrow\infty
}
}\sup_{\substack{N\in2^{\N}\\
TN^{-2}<\epsilon
}
}N^{-d}\norm{\chi_{[TN^{-2},\epsilon]}e^{it\De}\delta_{N}}_{L_{t,x}^{\infty}}=0,\label{eq:kernel 0 infty}
\end{equation}
which can be shown similar to \cite[(4.11)]{ionescu2012energy}; the
only key ingredient is (\ref{eq:Bourgain bound}).
\end{proof}
The next lemma states the weak scattering property on $\T^{d}$. For
later use in Proposition \ref{prop:amp}, in this version, we consider
general inhomogeneous evolutions, which are not necessarily solutions.
Due to the periodic resonance, we assume the time-convergence $t_{n}\rightarrow0$
of the frame and the negative-time linearity.
\begin{lem}
\label{lem:scatter T^d}Let $d\ge3$ and $C<\infty$. Let $\left\{ u_{n}\right\} $
be a bounded sequence in $C^{0}H^{1}(\R\times\T^{d})\cap Y^{1}$.
Let $\left\{ \mc O_{n}\right\} =\left\{ (N_{n},t_{n},x_{n})\right\} $
be a frame such that $t_{n}\rightarrow0$. If
\begin{equation}
\sup_{n}\norm{(i\d_{t}+\De)u_{n}}_{(Z^{-1})'}\le C\label{eq:scatter T^d f bound}
\end{equation}
and
\begin{equation}
u_{n}(t)=e^{it\De}u_{n}(0)\text{ holds for }t\le0,\label{eq:t<0 zero}
\end{equation}
then
\[
\left\{ e^{iT\De}\left(\iota_{\mc O_{n}}u_{n}\right)(-T)\right\} _{n\in\N}
\]
is uniformly convergent (in the weak sense) as $T\rightarrow\infty$.
\end{lem}

\begin{proof}
Explicitly, we show that for every $\phi\in\mc S(\R^{d})$,
\begin{equation}
\left\{ \jp{\phi,e^{iT\De}\left(\iota_{\mc O_{n}}u_{n}\right)(-T)}_{L^{2}(\R^{d})}\right\} _{n\in\N}\label{eq:scatter T^d claim0}
\end{equation}
is uniformly convergent as $T\rightarrow\infty$. Since the span of
$\left\{ \delta_{N_{*}}(\cdot-x_{*}):N_{*}\in2^{\Z}\text{ and }x_{*}\in\R^{d}\right\} $
is dense in $\mc S(\R^{d})$, one may assume further that 
\[
\phi=\delta_{N_{*}}(\cdot-x_{*}),\qquad N_{*}\in2^{\Z}\text{ and }x_{*}\in\R^{d}.
\]
Up to a comparable choice of the frame $\mc{\tilde O}_{n}=\left\{ (N_{n}N_{*}^{-1},t_{n},x_{n}+x_{*})\right\} $,
we can further reduce to the case $(N_{*},x_{*})=(1,0)$. Since $-T$
and $P_{1;\R^{d}}$ in the $\mc O_{n}$-coordinates correspond to
$t_{n}-TN_{n}^{-2}$ and $P_{N_{n};\T^{d}}$, respectively, in the
original domain, it suffices to show the uniform convergence as $T\rightarrow\infty$
of
\[
\left\{ \jp{\delta_{N_{n}}(\cdot-x_{n}),N_{n}^{1-\frac{d}{2}}e^{iTN_{n}^{-2}\De}u_{n}(t_{n}-TN_{n}^{-2})}_{L^{2}(\T^{d})}\right\} _{n\in\N}.
\]
We show the uniform convergence. For $n\in\N$ and $T_{1},T_{2}>0$,
denoting $f_{n}=(i\d_{t}+\De)u_{n}$, we have
\begin{align}
 & \left|\jp{\delta_{N_{n}}(\cdot-x_{n}),e^{iT_{1}N_{n}^{-2}\De}u_{n}(t_{n}-T_{1}N_{n}^{-2})-e^{iT_{2}N_{n}^{-2}\De}u_{n}(t_{n}-T_{2}N_{n}^{-2})}_{L^{2}(\T^{d})}\right|\nonumber \\
 & =\left|\jp{\delta_{N_{n}}(\cdot-x_{n}),\int_{t_{n}-T_{1}N_{n}^{-2}}^{t_{n}-T_{2}N_{n}^{-2}}e^{i(t_{n}-s)\De}f_{n}(s)ds}_{L^{2}(\T^{d})}\right|\nonumber \\
 & =\left|\int_{t_{n}-T_{1}N_{n}^{-2}}^{t_{n}-T_{2}N_{n}^{-2}}\jp{\delta_{N_{n}}(\cdot-x_{n}),e^{i(t_{n}-s)\De}f_{n}(s)ds}_{L^{2}(\T^{d})}\right|\nonumber \\
 & =\left|\int_{t_{n}-T_{1}N_{n}^{-2}}^{t_{n}-T_{2}N_{n}^{-2}}\jp{e^{i(s-t_{n})\De}\delta_{N_{n}}(\cdot-x_{n}),f_{n}(s)ds}_{L^{2}(\T^{d})}\right|\nonumber \\
 & \les\norm{\chi_{[-T_{1}N_{n}^{-2},-T_{2}N_{n}^{-2}]\cap[-t_{n},\infty)}e^{it\De}\delta_{N_{n}}(\cdot-x_{n})}_{Z^{-1}}\norm{f_{n}}_{(Z^{-1})'},\label{eq:scatter T^d last line}
\end{align}
where the restriction $[-t_{n},\infty)$ in (\ref{eq:scatter T^d last line})
is obtained from the condition (\ref{eq:t<0 zero}). Thus, we have
\begin{align*}
 & \left|N_{n}^{1-\frac{d}{2}}\jp{\delta_{N_{n}}(\cdot-x_{n}),e^{iT_{1}N_{n}^{-2}\De}u_{n}(t_{n}-T_{1}N_{n}^{-2})-e^{iT_{2}N_{n}^{-2}\De}u_{n}(t_{n}-T_{2}N_{n}^{-2})}_{L^{2}(\T^{d})}\right|\\
 & \les N_{n}^{1-\frac{d}{2}}\norm{\chi_{[-T_{1}N_{n}^{-2},-T_{2}N_{n}^{-2}]\cap[-t_{n},\infty)}e^{it\De}\delta_{N_{n}}(\cdot-x_{n})}_{Z^{-1}}\norm{f_{n}}_{(Z^{-1})'},
\end{align*}
which vanishes as $T_{1},T_{2}\rightarrow\infty$ by (\ref{eq:kernel 0})
and (\ref{eq:scatter T^d f bound}), finishing the proof.
\end{proof}

\subsection{\label{subsec:nonlinear conc}A nonlinear inverse property}

In this subsection, we show an inverse property of solutions to (\ref{eq:NLS T^d}).
The main difficulty is that the linear inverse property Proposition
\ref{prop:inverse embedding} is frequency-localized and does not
rule out functions split into small-amplitude bubbles of different
sizes. Although using strong multilinear estimates is a conventional
approach, such estimates have not been shown for the nonlinearity
when $d\ge5$.

Without a multilinear estimate, the following key proposition upgrades
Proposition \ref{prop:inverse embedding} to the concentration of
a nonlinear flow:
\begin{prop}
\label{prop:amp}Let $d\ge3$. Let $T_{n}\rightarrow0$ be a positive
sequence. Denote $I_{n}=[0,T_{n})$. Let $C<\infty$ and $\epsilon>0$.
Let $\left\{ u_{n}\right\} $ be a sequence of Duhamel solutions to
(\ref{eq:NLS T^d}) in $C^{0}H^{1}([0,T_{n}]\times\T^{d})\cap Y^{1}(I_{n})$.
Assume that
\begin{equation}
\sup_{n}\norm{u_{n}}_{Y^{1}(I_{n})}\le C\label{eq:conc condt1}
\end{equation}
and
\begin{equation}
\inf_{n}\norm{\chi_{I_{n}}\left(u_{n}-e^{it\De}u_{n}(0)\right)}_{Z^{1}}\ge\epsilon.\label{eq:conc condt3}
\end{equation}
Then, there exists a frame $\left\{ \mc O_{n}\right\} =\left\{ (N_{n},t_{n},x_{n})\right\} $
such that $t_{n}\in I_{n}$ and both
\[
\left\{ \chi_{\tilde{I_{n}}}\right\} \text{ and }\left\{ \iota_{\mc O_{n}}\left(u_{n}-e^{it\De}u_{n}(0)\right)(0)\right\} 
\]
are weakly nonzero, where $\tilde{I_{n}}=\left\{ N_{n}^{2}(t-t_{n})\mid t\in I_{n}\right\} \subset\R$
is the time interval mapped from $I_{n}$ by $\mc O_{n}$.
\end{prop}

\begin{proof}
Since the statement concerns only the time interval $I_{n}$, one
can equivalently regard $u_{n}$ as the cutoff solution $(u_{n},I_{n})$
to (\ref{eq:NLS T^d}). Let $v_{n}:=u_{n}-e^{it\De}u_{n}(0)$. For
$n\in\N$, (\ref{eq:conc condt3}) can be rewritten as
\begin{equation}
\sum_{N}\norm{\chi_{I_{n}}P_{N}v_{n}}_{Z^{1}}^{2}=\norm{\chi_{I_{n}}v_{n}}_{Z^{1}}^{2}\gtrsim1.\label{eq:r1}
\end{equation}
By (\ref{eq:conc condt1}), we have
\begin{equation}
\sum_{N}\norm{P_{N}v_{n}}_{Y^{1}}^{2}\les_{C}1.\label{eq:r2}
\end{equation}
By using (\ref{eq:cube Strichartz}) and $\norm{\mc N(u_{n})}_{L_{t,x}^{2}}\les\norm{u_{n}}_{L_{t,x}^{\frac{2}{1+a}}}^{1+a}=\norm{u_{n}}_{L_{t,x}^{\frac{2d+4}{d-2}}}^{1+a}$,
we have
\begin{align}
\sum_{N}\norm{\chi_{[-1,1)}P_{N}\d_{t}v_{n}}_{L^{\infty}H^{-1}+L_{t,x}^{2}}^{2} & =\norm{\chi_{[-1,1)}i\d_{t}v_{n}}_{\ell^{2}\left(L^{\infty}H^{-1}+L_{t,x}^{2}\right)}^{2}\label{eq:r2'}\\
 & \les\norm{\De v_{n}}_{\ell^{2}(L^{\infty}H^{-1})}^{2}+\norm{\chi_{[-1,1)}\mc N(u_{n})}_{L_{t,x}^{2}}^{2}\nonumber \\
 & \les\norm{v_{n}}_{\ell^{2}L^{\infty}H^{1}}^{2}+\norm{\chi_{[-1,1)}u_{n}}_{L_{t,x}^{\frac{2d+4}{d-2}}}^{2(1+a)}\nonumber \\
 & \les\norm{v_{n}}_{Y^{1}}^{2}+\norm{u_{n}}_{Y^{1}}^{2(1+a)}\les_{C}1.\nonumber 
\end{align}
By (\ref{eq:Nu bound}) and (\ref{eq:conc condt1}), we have
\begin{equation}
\sum_{N}\norm{\chi_{I_{n}}P_{N}\mc N(u_{n})}_{\left(Z^{-1}\right)'}^{2}=\norm{\chi_{I_{n}}\mc N(u_{n})}_{\left(Z^{-1}\right)'}^{2}\les\norm{u_{n}}_{Y^{1}}^{2(1+a)}\les_{C}1.\label{eq:r3}
\end{equation}
Let $\delta$ be a small number, e.g., $\delta=1/10$. We have
\begin{equation}
\sum_{N}N^{2\delta}\norm{P_{\ge N}u_{n}}_{Y^{1-\delta}}^{2}\les\sum_{\substack{N,K\\
K\ge N
}
}\left(N/K\right)^{2\delta}\norm{P_{K}u_{n}}_{Y^{1}}^{2}\les\sum_{K}\norm{P_{K}u_{n}}_{Y^{1}}^{2}\les\norm{u_{n}}_{Y^{1}}^{2}\les_{C}1\label{eq:r4}
\end{equation}
and similarly
\begin{equation}
\sum_{N}N^{-2\delta}\norm{P_{\le N}u_{n}}_{Y^{1+\delta}}^{2}\les_{C}1.\label{eq:r4'}
\end{equation}
Taking a sum of all the inequalities (\ref{eq:r2}), (\ref{eq:r2'}),
(\ref{eq:r3}), (\ref{eq:r4}), (\ref{eq:r4'}), and (\ref{eq:r1}),
we have
\begin{align*}
 & \sum_{N}\norm{\chi_{I_{n}}P_{N}v_{n}}_{Z^{1}}^{2}\gtrsim1\\
 & \gtrsim\sum_{N}\norm{P_{N}v_{n}}_{Y^{1}}^{2}+\norm{\chi_{[-1,1)}P_{N}\d_{t}v_{n}}_{L^{\infty}H^{-1}+L_{t,x}^{2}}^{2}+\norm{\chi_{I_{n}}P_{N}\mc N(u_{n})}_{\left(Z^{-1}\right)'}^{2}\\
 & +N^{2\delta}\norm{P_{\ge N}u_{n}}_{Y^{1-\delta}}^{2}+N^{-2\delta}\norm{P_{\le N}u_{n}}_{Y^{1+\delta}}^{2},
\end{align*}
thus there exists $N_{n}\in2^{\N}$ such that
\begin{align}
\norm{\chi_{I_{n}}P_{N_{n}}v_{n}}_{Z^{1}} & \gtrsim\norm{P_{N_{n}}v_{n}}_{Y^{1}}\label{eq:rbound Y1}\\
 & +\norm{\chi_{[-1,1)}P_{N_{n}}\d_{t}v_{n}}_{L^{\infty}H^{-1}+L_{t,x}^{2}}\label{eq:rbound dt}\\
 & +\norm{\chi_{I_{n}}P_{N_{n}}\mc N(u_{n})}_{\left(Z^{-1}\right)'}\label{eq:rbound Nu}\\
 & +N_{n}^{\delta}\norm{P_{\ge N_{n}}u_{n}}_{Y^{1-\delta}}+N_{n}^{-\delta}\norm{P_{\le N_{n}}u_{n}}_{Y^{1+\delta}}.\label{eq:rbound scale break}
\end{align}
Denote $\al_{n}:=\norm{\chi_{I_{n}}P_{N_{n}}v_{n}}_{Z^{1}}$. By (\ref{eq:rbound Y1}),
we have $\norm{\chi_{I_{n}}\al_{n}^{-1}P_{N_{n}}v_{n}}_{Z^{1}}=1\gtrsim\norm{\al_{n}^{-1}P_{N_{n}}v_{n}}_{Y^{1}}$,
thus by Proposition \ref{prop:inverse embedding}, there exists a
frame $\left\{ \mc O_{n}\right\} =\left\{ (N_{n},t_{n},x_{n})\right\} $
with $t_{n}\in I_{n}$ such that $\left\{ \iota_{\mc O_{n}}\left(\al_{n}^{-1}P_{N_{n}}v_{n}\right)(0)\right\} $
is weakly nonzero. Passing to a subsequence, assume that
\begin{equation}
\iota_{\mc O_{n}}\left(\al_{n}^{-1}P_{N_{n}}v_{n}\right)\weak v_{*}.\label{eq:weak to v*}
\end{equation}
By (\ref{eq:rbound dt}), $\left\{ \iota_{\mc O_{n}}\left(\al_{n}^{-1}P_{N_{n}}v_{n}\right)\right\} $
is equicontinuous in $H_{\loc}^{-1}(\R^{d})$, thus we may assume
further that for every $T\in\R$,
\[
\iota_{\mc O_{n}}\left(\al_{n}^{-1}P_{N_{n}}v_{n}\right)(T)\weak v_{*}(T).
\]
In particular, since $\left\{ \iota_{\mc O_{n}}\left(\al_{n}^{-1}P_{N_{n}}v_{n}\right)(0)\right\} $
is weakly nonzero, we have $v_{*}(0)\neq0$.

If $\inf\al_{n}>0$ and $\left\{ \chi_{\tilde{I_{n}}}\right\} $ is
weakly nonzero, then 
\[
P_{1}\left(\iota_{\mc O_{n}}v_{n}\right)(0)=\al_{n}\cdot\iota_{\mc O_{n}}\left(\al_{n}^{-1}P_{N_{n}}v_{n}\right)(0)
\]
is also weakly nonzero thus there is nothing to prove. Thus, passing
to a subsequence, we assume either $\al_{n}\rightarrow0$ or $\chi_{\tilde{I_{n}}}\weak0$.
In the rest of this proof, we show a contradiction. By (\ref{eq:rbound scale break}),
we have
\begin{align*}
\al_{n}=\norm{\chi_{I_{n}}P_{N_{n}}v_{n}}_{Z^{1}} & \gtrsim N_{n}^{\delta}\norm{P_{\ge N_{n}}u_{n}}_{Y^{1-\delta}}+N_{n}^{-\delta}\norm{P_{\le N_{n}}u_{n}}_{Y^{1+\delta}}\\
 & \gtrsim N_{n}^{\delta}\norm{P_{\ge N_{n}}u_{n}}_{L^{\infty}L^{2_{1-\delta}}}+N_{n}^{-\delta}\norm{P_{\le N_{n}}u_{n}}_{L^{\infty}L^{2_{1+\delta}}}
\end{align*}
hence $\al_{n}^{-1}\iota_{\mc O_{n}}u_{n}$ is bounded in $L^{\infty}(L^{2_{1-\delta}}+L^{2_{1+\delta}})_{\loc}\hook L^{\infty}L_{\loc}^{1+a}$,
where we used $2_{1+\delta}>2_{1-\delta}>\frac{d+2}{d-2}=1+a$. Thus,
\[
\al_{n}^{-1-a}\iota'_{\mc O_{n}}\left(P_{N_{n}}\mc N(u_{n})\right)=\al_{n}^{-1-a}P_{1}\mc N(\iota_{\mc O_{n}}u_{n})
\]
is bounded in $L^{\infty}L_{\loc}^{1}$, then since we assumed either
$\al_{n}\rightarrow0$ or $\chi_{\tilde{I_{n}}}\weak0$, we have
\begin{align*}
(i\d_{t}+\De)\iota_{\mc O_{n}}\left(\al_{n}^{-1}P_{N_{n}}v_{n}\right) & =\iota'_{\mc O_{n}}\left(\al_{n}^{-1}\chi_{I_{n}}P_{N_{n}}\mc N(u_{n})\right)\\
 & =\al_{n}^{a}\chi_{\tilde{I_{n}}}\cdot\al_{n}^{-1-a}\iota'_{\mc O_{n}}\left(P_{N_{n}}\mc N(u_{n})\right)\weak0.
\end{align*}
Thus, $v_{*}$, which is the weak limit of $\iota_{\mc O_{n}}\left(\al_{n}^{-1}P_{N_{n}}v_{n}\right)$,
is a distributional solution to $(i\d_{t}+\De)v_{*}=0$. By (\ref{eq:rbound Y1})
and the embedding $Y^{1}\hook L^{\infty}H^{1}$, we have $v_{*}\in L^{\infty}\dot{H}^{1}$.
Thus, the spatial Fourier transform can be applied to $v_{*}$, from
which we obtain that $v_{*}$ is a (Duhamel) free evolution almost
everywhere. By Lemma \ref{lem:scatter T^d} and (\ref{eq:rbound Nu}),
we have
\[
\lim_{T\rightarrow\infty}\lim_{n\rightarrow\infty}e^{iT\De}\iota_{\mc O_{n}}\left(\al_{n}^{-1}P_{N_{n}}v_{n}\right)(-T)=\lim_{n\rightarrow\infty}\lim_{T\rightarrow\infty}e^{iT\De}\iota_{\mc O_{n}}\left(\al_{n}^{-1}P_{N_{n}}v_{n}\right)(-T),
\]
the left-hand side of which equals to $\lim_{T\rightarrow\infty}e^{iT\De}v_{*}(-T)$,
while the right-hand side of which equals to $0$ since $v_{n}(t)=0$
holds for $t<0$. This contradicts that $v_{*}$ is a free evolution
and finishes the proof.
\end{proof}

\subsection{\label{sec:Proof-of-Theorem}Proof of Theorem \ref{thm:defocusing GWP}
and Theorem \ref{thm:focusing GWP}}

In this section, we prove Theorem \ref{thm:defocusing GWP} and Theorem
\ref{thm:focusing GWP}. Indeed, the proofs are identical; for simplicity,
we prove only the defocusing case, Theorem \ref{thm:defocusing GWP}.
\begin{prop}
\label{prop:main}Let $d\ge3$ and $E_{0}<\infty$. There is no sequence
$u_{n}\in C^{0}H^{1}\cap Y^{1}([0,T_{n,\max}))$ of solutions to (\ref{eq:NLS T^d})
with the defocusing sign and maximal positive lifespans $T_{n,\max}\rightarrow0$
such that $E(u_{n}(0))\le E_{0}$ and $\norm{u_{n}}_{Z^{1}}=\infty$.
\end{prop}

Proposition \ref{prop:main} is actually stronger than Theorem \ref{thm:defocusing GWP}.
Once we show Proposition \ref{prop:main}, Theorem \ref{thm:defocusing GWP}
will follow as shown below.
\begin{proof}[Proof of Theorem \ref{thm:defocusing GWP} assuming Proposition \ref{prop:main}]
Let $u\in C^{0}H^{1}\cap Y^{1}$ be a Duhamel solution to (\ref{eq:NLS T^d})
with initial data $u(0)=u_{0}$ and the positive lifespan  $[0,T_{\max})$.
If $T_{\max}=\infty$, there is nothing to prove. Assume $T_{\max}<\infty$.
Let $E_{0}:=E(u_{0})$. By $T_{\max}<\infty$ and Proposition \ref{prop:atomic space props},
we have $\norm u_{Y^{1}}=\infty$. By (\ref{eq:Nu bound}), we have
\[
\norm u_{Y^{1}}=\norm{e^{it\De}u_{0}+K^{+}\mc N(u)}_{Y^{1}}\les\norm{e^{it\De}u_{0}}_{H^{1}}+\norm{\mc N(u)}_{(Z^{-1})'}\les_{E_{0}}1+\norm u_{Z^{1}}^{1+a},
\]
which implies $\norm u_{Z^{1}}=\infty$. Applying Proposition \ref{prop:main}
to the sequence $u_{n}(t):=u(t+T_{\max}-1/n)$, we obtain a contradiction,
finishing the proof.
\end{proof}
\begin{proof}[Proof of Proposition \ref{prop:main}]
Assume $\left\{ u_{n}\right\} $ is such a sequence. Since $\norm{u_{n}}_{Z^{1}([0,T_{n,\max}))}=\infty$,
there exist $0=T_{n,0}<T_{n,1}<T_{n,2}<\cdots<T_{n,\max}$ such that
\begin{equation}
\norm{\chi_{I_{n,j}}\left(u_{n}-e^{i(t-T_{n,j-1})\De}u(T_{n,j-1})\right)}_{Z^{1}}=1,\label{eq:Z1>1}
\end{equation}
where we denoted $I_{n,j}=[T_{n,j-1},T_{n,j})$. For each $j\in\N$,
by (\ref{eq:Nu bound}), we have
\begin{equation}
\norm{\chi_{I_{n,j}}u_{n}}_{Y^{1}}\les\norm{u_{n}(T_{n,j-1})}_{H^{1}}+\norm{\chi_{I_{n,j}}\mc N(u_{n})}_{(Z^{-1})'}\les_{E_{0}}1+\norm{\chi_{I_{n,j}}u_{n}}_{Z^{1}}^{1+a}\les_{j}1.\label{eq:Y1<1}
\end{equation}
For $j,n\in\N$, let $u_{n,j},u_{n,\le j}\in C^{0}H^{1}\cap Y^{1}$
be the cutoff solutions to (\ref{eq:NLS T^d}) extending $u\mid_{[T_{n,j-1},T_{n,j})}$
and $u\mid_{[0,T_{n,j})}$, respectively. By (\ref{eq:Y1<1}), (\ref{eq:Z1>1}),
and the condition $T_{n,\max}\rightarrow0$, applying Proposition
\ref{prop:amp} to $\left\{ u_{n,j}(\cdot-T_{n,j-1})\right\} _{n,j\in\N}$
yields the existence of a frame $\left\{ \mc O_{n,j}\right\} =\left\{ (N_{n,j},t_{n,j},x_{n,j})\right\} $
such that $t_{n,j}\in I_{n,j}$ and both
\[
\left\{ \chi_{\tilde{I_{n,j}}}\right\} _{n,j\in\N}
\]
and
\begin{equation}
\left\{ \iota_{\mc O_{n,j}}\left(u_{n,j}-e^{i(t-T_{n,j-1})\De}u_{n}(T_{n,j-1})\right)(0)\right\} _{n,j\in\N}=\left\{ \iota_{\mc O_{n,j}}\left(u_{n,\le j}-u_{n,\le j-1}\right)(0)\right\} _{n,j\in\N}\label{eq:long}
\end{equation}
are weakly nonzero, where $\tilde{I_{n,j}}=\left\{ N_{n,j}^{2}(t-t_{n,j})\mid t\in I_{n,j}\right\} $
denotes the time interval mapped from $I_{n,j}$ by $\mc O_{n,j}$.
Thus, passing to a subsequence of $n$, either 
\[
\left\{ \iota_{\mc O_{n,j}}u_{n,\le j}(0)\right\} _{n,j\in\N}\text{ or }\left\{ \iota_{\mc O_{n,j}}u_{n,\le j-1}(0)\right\} _{n,j\in\N}
\]
is weakly nonzero. The proof is identical for either case; we assume
the former for simplicity.

Passing to a subsequence of $n$, by Lemma \ref{lem:weak lim is sol},
we have the weak convergence for each $j\in\N$: there exists a cutoff
solution $(v_{\le j},I_{\le j})$, $v_{\le j}\in C^{0}\dot{H}^{1}\cap L_{t,x}^{\frac{2d+4}{d-2}}(\R\times\R^{d})$
to (\ref{eq:NLS R^d}) such that for each $T\in\R$,
\begin{equation}
\left(\iota_{\mc O_{n,j}}u_{n,\le j}\right)(T)\weak v_{\le j}(T).\label{eq:->v(T)}
\end{equation}
By Proposition \ref{prop:defocusing R^d} on $v_{\le j}\mid_{I_{\le j}}$
and the Strichartz estimate on $v_{\le j}\mid_{\R\setminus I_{\le j}}$,
we have an a priori bound
\begin{equation}
\sup_{j\in\N}\norm{v_{\le j}}_{C^{0}\dot{H}^{1}\cap L_{t,x}^{\frac{2d+4}{d-2}}}\les_{E_{0}}1.\label{eq:vj bound}
\end{equation}
By (\ref{eq:vj bound}) and Lemma \ref{lem:weak lim is sol}, there
exists a cutoff solution $(v_{*},I_{*})$, $v_{*}\in C^{0}\dot{H}^{1}\cap L_{t,x}^{\frac{2d+4}{d-2}}(\R\times\R^{d})$
to (\ref{eq:NLS R^d}) that scatters and a subsequence $\left\{ j_{k}\right\} $
such that for $T\in\R$, 
\[
v_{\le j_{k}}(T)\weak v_{*}(T)
\]
holds. Since $\left\{ \iota_{\mc O_{n,j}}u_{n,\le j}(0)\right\} _{n,j\in\N}$
is weakly nonzero, so is $\left\{ v_{\le j}(0)\right\} _{j\in\N}$,
thus $v_{*}(0)\neq0$.

We evaluate the following limit in two ways:
\begin{equation}
\lim_{k\rightarrow\infty}\lim_{n\rightarrow\infty}\lim_{T\rightarrow\infty}e^{iT\De}\left(\iota_{\mc O_{n,j_{k}}}u_{n,\le j_{k}}\right)(-T).\label{eq:limlimlim}
\end{equation}
Since $\left\{ u_{n,\le j}\right\} _{n\in\N}$ is bounded in $Y^{1}$
for each $j\in\N$ and $\left\{ v_{\le j}\right\} $ is bounded in
$C^{0}\dot{H}^{1}\cap L_{t,x}^{\frac{2d+4}{d-2}}$, by Lemma \ref{lem:scatter T^d}
and Lemma \ref{lem:scatter R^d}, (\ref{eq:limlimlim}) equals
\begin{equation}
\lim_{k\rightarrow\infty}\lim_{T\rightarrow\infty}e^{iT\De}v_{\le j_{k}}(-T)=\lim_{T\rightarrow\infty}e^{iT\De}v_{*}(-T),\label{eq:limlimlim 1}
\end{equation}
which exists and is nonzero since $v_{*}\neq0$ scatters.

Since $u_{n,\le j}=e^{it\De}u_{n}(0)$ holds for $t<0$ and $\iota_{\mc O_{n,j_{k}}}\left(e^{it\De}u_{n}(0)\right)$
is a free evolution, (\ref{eq:limlimlim}) can also be rewritten as
\begin{equation}
\lim_{k\rightarrow\infty}\lim_{n\rightarrow\infty}\lim_{T\rightarrow\infty}e^{iT\De}\iota_{\mc O_{n,j_{k}}}\left(e^{it\De}u_{n,\le j_{k}}(0)\right)(-T)=\lim_{k\rightarrow\infty}\lim_{n\rightarrow\infty}\iota_{\mc O_{n,j_{k}}}\left(e^{it\De}u_{n}(0)\right)(0).\label{eq:limlimlim 2}
\end{equation}
Thus, we have
\begin{equation}
\lim_{k\rightarrow\infty}\lim_{n\rightarrow\infty}\iota_{\mc O_{n,j_{k}}}\left(e^{it\De}u_{n}(0)\right)(0)\neq0.\label{eq:!=00003D0}
\end{equation}
By (\ref{eq:!=00003D0}) and the weak nonzeroness of $\left\{ \chi_{\tilde{I_{n,j_{k}}}}\right\} _{n,k\in\N}\subset\left\{ \chi_{\tilde{I_{n,j}}}\right\} _{n,j\in\N}$,
the family of time-cutoffs of linear evolutions 
\[
\left\{ \iota_{\mc O_{n,j_{k}}}\left(\chi_{I_{n,j_{k}}}e^{it\De}u_{n}(0)\right)\right\} _{n,k\in\N}=\left\{ \chi_{\tilde{I_{n,j_{k}}}}\iota_{\mc O_{n,j_{k}}}\left(e^{it\De}u_{n}(0)\right)\right\} _{n,k\in\N}
\]
is also weakly nonzero. Thus, we have the critical norm bound
\begin{equation}
\liminf_{k\rightarrow\infty}\liminf_{n\rightarrow\infty}\norm{\chi_{I_{n,j_{k}}}e^{it\De}u_{n}(0)}_{L_{t,x}^{\frac{2d+4}{d-2}}}>0,\label{eq:lim>0}
\end{equation}
By (\ref{eq:lim>0}) and the Fatou's lemma, we have
\begin{align*}
 & \liminf_{n\rightarrow\infty}\norm{e^{it\De}u_{n}(0)}_{L_{t,x}^{\frac{2d+4}{d-2}}([0,1)\times\T^{d})}\\
 & =\liminf_{n\rightarrow\infty}\norm{\norm{\chi_{I_{n,j}}e^{it\De}u_{n}(0)}_{L_{t,x}^{\frac{2d+4}{d-2}}([0,1)\times\T^{d})}}_{\ell_{j}^{\frac{2d+4}{d-2}}}\\
 & \ge\norm{\liminf_{n\rightarrow\infty}\norm{\chi_{I_{n,j_{k}}}e^{it\De}u_{n}(0)}_{L_{t,x}^{\frac{2d+4}{d-2}}([0,1)\times\T^{d})}}_{\ell_{k}^{\frac{2d+4}{d-2}}}=\infty,
\end{align*}
which contradicts (\ref{eq:Bourgain Strichartz}). Thus, the assumption
of this proposition cannot hold, and we finish the proof.
\end{proof}

\appendix

\section{Reduction to almost periodic solutions\label{subsec:Adaptation-to-the}}

So far we have focused only on showing the global well-posedness of
NLS on $\T^{d}$ presuming that on $\R^{d}$. However, the argument
in Proposition \ref{prop:main} can also be used to reduce the GWP
problem on $\R^{d}$, Proposition \ref{prop:defocusing R^d}, to the
nonexistence of $\dot{H}^{1}(\R^{d})$ solution that is almost periodic
modulo scaling. This is conventionally reduced to showing the following
Palais-Smale criterion:
\begin{claim}[{Palais-Smale criterion, \cite[Proposition 3.1]{killip2010focusing}}]
\label{claim:Palais Smale}Let $d\ge3$. Assume Proposition \ref{prop:defocusing R^d}
fails. Let $E_{\inf}$ be the infimum of $E_{0}$ for which Proposition
\ref{prop:defocusing R^d} fails. Let $\left\{ u_{n}\right\} $ be
a sequence of solutions $u_{n}\in C^{0}\dot{H}^{1}\cap L_{t,x}^{\frac{2d+4}{d-2}}((T_{n}^{-},T_{n}^{+})\times\R^{d})$,
$T_{n}^{-}<0<T_{n}^{+}$ to (\ref{eq:NLS R^d}) such that
\[
\limsup_{n\rightarrow\infty}E(u_{n})=E_{\inf}
\]
and
\begin{equation}
\lim_{n\rightarrow\infty}\norm{u_{n}}_{L_{t,x}^{\frac{2d+4}{d-2}}([0,T_{n}^{+})\times\R^{d})}=\lim_{n\rightarrow\infty}\norm{u_{n}}_{L_{t,x}^{\frac{2d+4}{d-2}}((T_{n}^{-},0]\times\R^{d})}=\infty.\label{eq:unbounded}
\end{equation}
Then, there exist sequences $\left\{ N_{n}\right\} \subset2^{\Z}$
and $\left\{ x_{n}\right\} \subset\R^{d}$ such that $\left\{ N_{n}^{1-\frac{d}{2}}u_{n}(0,N_{n}^{-1}\cdot+x_{n})\right\} $
converges in $\dot{H}^{1}(\R^{d})$ along a subsequence.
\end{claim}

The original proof of Claim \ref{claim:Palais Smale} in \cite[Proposition 3.1]{killip2010focusing}
used a linear profile decomposition and stability theory.\footnote{In \cite[Proposition 3.1]{killip2010focusing}, the kinetic energy
(i.e., $\norm{\na u_{n}}_{C^{0}L^{2}}^{2}$) replaced the role of
the energy in Claim \ref{claim:Palais Smale}. Either choice leads
to the reduction of GWP to the almost periodic case by the same convergence
argument.} In particular, a nontrivial H\"older continuity of the flow map \cite{tao2005stability}
was employed.

We provide an alternative proof of Claim \ref{claim:Palais Smale},
which mainly follows the argument of Proposition \ref{prop:main}.
For the proof, we note that all the analysis we did on $\T^{d}$ work
the same on $\R^{d}$. On $\R^{d}$, there is no short-time restriction
(since we no longer have resonance), frequency sizes are $N_{n}\in2^{\Z}$,
and $\iota_{\mc O_{n,j}}$ is simply a scaling operator. Although
we show only the defocusing case, note that the same works for the
focusing case due to the lower semicontinuity of the energy functional
$E$, based on a Fatou property below the threshold energy.
\begin{proof}
Let $\left\{ \mc O_{n,j}\right\} =\left\{ (N_{n,j},t_{n,j},x_{n,j})\right\} $
(defined for $n\gg_{j}1$) and $v_{\le j}$ be as in the proof of
Proposition \ref{prop:main}. Since (\ref{eq:vj bound}) is the only
place where Proposition \ref{prop:defocusing R^d} is used, we have
$\limsup_{j}E(v_{\le j})\ge E_{\inf}$. Thus, for each $\epsilon>0$,
there exists $j\in\N$ such that $E(v_{\le j})\ge E_{\inf}-\epsilon$.
Assume $\epsilon\ll1$. Up to scaling, we may assume $N_{n,j}=1$
and $x_{n,j}=0$. (\ref{eq:->v(T)}) can be rewritten as $u_{n,\le j}(t_{n,j}+T)\weak v_{\le j}(T)$
for $T\in\R$. Thus, we have
\begin{equation}
\limsup_{n\rightarrow\infty}\norm{u_{n,\le j}(t_{n,j}+T)-v_{\le j}(T)}_{\dot{H}^{1}}^{2}\les E_{\inf}-E(v_{\le j})\les\epsilon.\label{eq:Fatou remark}
\end{equation}
If $t_{n,j}\rightarrow\infty$ and $v_{\le j}$ scatters backward
in time, there exists $T>-\infty$ such that $\norm{e^{it\De}v_{\le j}(T)}_{L_{t,x}^{\frac{2d+4}{d-2}}((-\infty,0]\times\R^{d})}\ll1$,
then by (\ref{eq:Fatou remark}) and standard small-data theory, $u_{n}$
scatters backward in time with $\limsup_{n\rightarrow\infty}\norm{u_{n}}_{L_{t,x}^{\frac{2d+4}{d-2}}((-\infty,t_{n,j}+T]\times\R^{d})}\ll1$,
contradicting (\ref{eq:unbounded}).

If $t_{n,j}\rightarrow\infty$ and $v_{\le j}$ does not scatter backward
in time, we have
\[
\limsup_{n\rightarrow\infty}\norm{u_{n,\le j}}_{L_{t,x}^{\frac{2d+4}{d-2}}(\R\times\R^{d})}\ge\norm{v_{\le j}}_{L_{t,x}^{\frac{2d+4}{d-2}}((-\infty,0]\times\R^{d})}=\infty,
\]
contradicting the norm bound (\ref{eq:Y1<1}) and the embedding $Y^{1}\hook L_{t,x}^{\frac{2d+4}{d-2}}$.

Thus, passing to a subsequence, we may assume $t_{n,j}\rightarrow T_{*}<\infty$.
By (\ref{eq:->v(T)}), we have $u_{n,\le j}(t_{n,j}-T_{*})\weak v_{\le j}(-T_{*})$.
Since $\d_{t}u_{n,\le j}$ is a time-cutoff of $\d_{t}u_{n}$, which
is bounded in $L^{\infty}H_{\loc}^{-1}$, $\left\{ u_{n,\le j}\right\} _{n\in\N}$
is equicontinuous in $H_{\loc}^{-1}(\R^{d})$ and thus $u_{n}(0)=u_{n,\le j}(0)\weak v_{\le j}(-T_{*})$
also holds true.

Now removing our simplification $N_{n,j}=1$ and $x_{n,j}=0$ above,
what we obtained is that $\left\{ N_{n,j}^{1-\frac{d}{2}}u_{n}(0,N_{n,j}^{-1}\cdot+x_{n,j})\right\} _{n\in\N}$,
which is a sequence of $\dot{H}^{1}(\R^{d})$-critical rescales of
$\left\{ u_{n}(0)\right\} _{n\in\N}$, converges weakly to $v_{\le j}(-T_{*})$,
which has the energy $E(v_{\le j}(-T_{*}))\ge E_{\inf}-\epsilon$.
Taking $\epsilon\rightarrow0$ finishes the proof.
\end{proof}
\bibliographystyle{plain}
\bibliography{citationforTd}
\end{document}